\documentclass[11pt, notitlepage]{article}
\usepackage{amssymb,amsmath,comment}
\catcode`\@=11 \@addtoreset{equation}{section}
\def\thesection{\arabic{section}}

\def\theequation{\thesection.\arabic{equation}}
\catcode`\@=12
\usepackage{colortbl}%
\usepackage{a4wide}
\usepackage{parskip}
\newcommand{\ds} {\displaystyle}
\newcommand{\e}{\epsilon}

\newcommand{\pa} {\partial}
\newcommand{\al} {\alpha}
\newcommand{\ba} {\beta}
\newcommand{\ga} {\gamma}
\newcommand{\Ga} {\Gamma}
\newcommand{\Om} {\Omega}
\newcommand{\ra} {\rightarrow}
\newcommand{\ov}{\overline}
\newcommand{\de} {\delta}
\newcommand{\De} {\Delta}
\newcommand{\la} {\lambda}
\newcommand{\La} {\Lambda}
\newcommand{\vth}{\vartheta}
\newcommand{\noi} {\noindent}
\newcommand{\na} {\nabla}
\newcommand{\ul} {\underline}
\newcommand{\mb} {\mathbb}
\newcommand{\mc} {\mathcal}

\newcommand{\ld} {\langle}
\newcommand{\rd} {\rangle}
\newcommand{\tl} {\tilde}

\usepackage[all]{xy}
\catcode`\@=11
\def\theequation{\@arabic{\c@section}.\@arabic{\c@equation}}
\catcode`\@=12

\def\QED{\hfill {$\square$}\goodbreak \medskip}

\newtheorem{Theorem}{Theorem}[section]
\newtheorem{Lemma}[Theorem]{Lemma}
\newtheorem{Proposition}[Theorem]{Proposition}
\newtheorem{Corollary}[Theorem]{Corollary}
\newtheorem{Remark}[Theorem]{Remark}
\newtheorem{Definition}[Theorem]{Definition}

\def\Xint#1{\mathchoice
	{\XXint\displaystyle\textstyle{#1}}%
	{\XXint\textstyle\scriptstyle{#1}}%
	{\XXint\scriptstyle\scriptscriptstyle{#1}}%
	{\XXint\scriptscriptstyle\scriptscriptstyle{#1}}%
	\!\int}
\def\XXint#1#2#3{{\setbox0=\hbox{$#1{#2#3}{\int}$ }
		\vcenter{\hbox{$#2#3$ }}\kern-.6\wd0}}

\begin{document}
	{\vspace{0.01in}
		\title
		{A qualitative study of $(p,q)$ Singular parabolic equations: local existence, Sobolev regularity and asymptotic behaviour}

		\author{  Jacques Giacomoni$^{\,1}$ \footnote{e-mail: {\tt jacques.giacomoni@univ-pau.fr}}, \ Deepak Kumar$^{\,2}$\footnote{e-mail: {\tt deepak.kr0894@gmail.com}},  \
			and \  K. Sreenadh$^{\,2}$\footnote{
				e-mail: {\tt sreenadh@maths.iitd.ac.in}} \\
			\\ $^1\,${\small Universit\'e  de Pau et des Pays de l'Adour, LMAP (UMR E2S-UPPA CNRS 5142) }\\ {\small Bat. IPRA, Avenue de l'Universit\'e F-64013 Pau, France}\\  
			$^2\,${\small Department of Mathematics, Indian Institute of Technology Delhi,}\\
			{\small	Hauz Khaz, New Delhi-110016, India } }

		\date{}
		
		\maketitle

\begin{abstract}
 The purpose of the article is to study the existence, regularity, stabilization and blow up results of weak solution to the following parabolic $(p,q)$-singular equation:
	\begin{equation*}
	  (P_t)\; \left\{\begin{array}{rllll}
	    u_t-\Delta_{p}u -\Delta_{q}u & = \vth \; u^{-\de}+ f(x,u), \; u>0 \text{ in } \Om\times (0,T), \\ u&=0 \quad \text{ on }  \pa\Om\times (0,T), \\
	    u(x,0)&= u_0(x) \; \text{ in }\Om,
	   \end{array}
		\right.
	\end{equation*}
  where $\Om$ is a bounded domain in $\mathbb{R}^N$ with $C^2$ boundary $\pa\Om$, $1<q<p< \infty$, $0<\de, T>0$, $N\ge 2$ and $\vth>0$ is a parameter. Moreover, we assume that $f:\Om\times [0,\infty) \to \mb R$ is a bounded below Carath\'eodory function, locally Lipschitz with respect to the second variable uniformly in $x\in\Om$ and $u_0\in L^\infty(\Om)\cap W^{1,p}_0(\Om)$. We distinguish the cases as $q$-subhomogeneous and $q$-superhomogeneous depending on the growth of $f$ (hereafter we will drop the term $q$). In the subhomogeneous case,  we prove the existence and uniqueness of the weak solution to problem $(P_t)$ for $\de<2+1/(p-1)$. For this, we first study the stationary problems corresponding to $(P_t)$ by using  the method of sub and super solutions and subsequently employing implicit Euler method, we obtain the existence of a solution to $(P_t)$. Furthermore, in this case,
  we prove the stabilization result, that is, the solution $u(t)$ of $(P_t)$ converges to $u_\infty$, the unique solution to the stationary problem, in $L^\infty(\Om)$ as $t\ra\infty$.  For the superhomogeneous case, we prove the local existence theorem by taking help of nonlinear semigroup theory. Subsequently, we prove finite time blow up of solution to problem $(P_t)$ for small parameter $\vartheta>0$ in the case $\de\leq 1$ and for all $\vth>0$ in the case $\de>1$. Moreover, we prove higher Sobolev integrability of the solution to purely singular problem corresponding to the steady state of $(P_t)$, which is of independent interest. As a consequence of this, we  improve the regularity of solution to $(P_t)$ for the case $\de<2+1/(p-1)$.
  \medskip
			
 \noi \textbf{Key words:} $(p,q)$-Parabolic equations, non-homogeneous elliptic operator, singular nonlinearity, existence and stabilization of weak solution, nonlinear semigroup theory, blow up result,  regularity results.
 \medskip
 
  \noi \textit{2010 Mathematics Subject Classification:} 35K92, 35K59, 35J20, 35J62
		
\end{abstract}
  \bigskip

\section{Introduction}
  Let $\Om$ be a bounded domain in $\mathbb{R}^N$ $(N\ge 2)$, with $C^2$ boundary $\pa\Om$, and let $T>0$ with $Q_T:=\Om\times (0,T)$ and $\Ga_T:=\pa\Om\times (0,T)$. In this paper, we are interested in the study of following $(p,q)$-parabolic equation involving singular nonlinearity,
  \begin{equation*}
    (P_t)\; \left\{\begin{array}{rllll}
     u_t-\Delta_{p}u -\Delta_{q}u & = \vth \; u^{-\de}+ f(x,u), \; u>0 \text{ in } Q_T, \\ u&=0 \quad \text{ on }  \Ga_T, \\
     u(x,0)&= u_0(x) \; \text{ in }\Om,
    \end{array}
    \right.
  \end{equation*}
 where $1<q<p< \infty$, $\de>0$ and $\vth>0$ is a parameter. Furthermore, we assume that $f:\Om\times [0,\infty) \to \mb R$ is a bounded below Carath\'eodory function, the map $s\mapsto f(x,s)$ is locally Lipschitz uniformly in $x\in\Om$ and $u_0\in L^\infty(\Om)\cap W^{1,p}_0(\Om)$ is a nonnegative function. $\Delta_{p}$ is the $p$-Laplace operator, defined as 
 $\Delta_{p} u= \nabla\cdot(|\nabla u|^{p-2}\nabla u)$.
 The operator $A_{p,q}:=-\Delta_{p}-\Delta_{q}$ is known as $(p,q)$-Laplacian.\par 
 The motivation to study problem $(P_t)$ comes from its large variety of applications in the physical world. For instance, when $p=q=2$, it has applications in non-Newtonian fluids, in particular pseudoplastic fluids, in boundary layer phenomena for viscous fluids, in the Langmuir-Hinshelwood model of chemical heterogeneous catalyst kinetics, in enzymatic kinetic models, as well as in the theory of heat conduction in electrically conducting materials and in the study of guided modes of an electromagnetic field in the nonlinear medium. For $p\neq 2$, it arises from the study of turbulent flow of gas in porous media, see \cite{badra} for a brief introduction to this topic. For further details, we refer to the survey of Hern\'andez-Mancebo \cite{hernand} and the book of Ghergu-R\v adulescu \cite{ghergu}. 
 For $p\neq q$, that is, the equations of kind $(P_t)$  arise in the form of general reaction-diffusion equation:
 \begin{equation}\label{prb1}
 u_t= \mathrm{div} [A(u)\nabla u]+ r(x,u),
 \end{equation}
 where $A(u)= |\nabla u|^{p-2}+|\nabla u|^{q-2}$. The problem \eqref{prb1} has wide range of important applications in biophysics, plasma physics and chemical reactions, where the function $u$ corresponds to the concentration term, the first term on the right hand side represents diffusion with a diffusion coefficient $A(u)$ and the second term is the reaction which relates to sources and loss processes.  For more details,  readers are referred to  \cite{cherfils, maranorcn} and references therein. \par
 The evolution equations of type $(P_t)$ and corresponding stationary problems have been explored extensively by researchers in recent decades when the differential operator is homogeneous in nature, that is, $p=q$. In particular, when $p=q=2$, the seminal work of Crandall et al. in \cite{crandall} is considered as the beginning of vast research on stationary problems involving singular nonlinearity. Interested reader may refer to the work of Boccardo and Orsina \cite{bocardo} and the bibliography of Hern\'andez and Mancebo \cite{hernand}, and for the case $p=q\neq 2$ to the work of Giacomoni et al. \cite{bough,giacomoni}. The case of stationary problems with singular nonlinearity and nonhomogeneous elliptic operators have very recent history, for instance see the contributions of R\v adulescu et al. \cite{dpk,papageorg} for the case $\de<1$  and of Giacomoni et al. \cite{JDS} for all $\de>0$.\par
 Turning to the parabolic problems involving singular nonlinearity, for $p=q=2$, Tak\'a\v{c} in \cite{takac}, studied the stabilization results for the semilinear parabolic problem of the type $(P_t)$, that is, the solution $u(t)$ converges to the solution of steady state problem in $C^1(\ov\Om)$, as $t\ra\infty$. For the case $p\neq 2$, Badra et al. in \cite{badra}, De Bonis et al. in \cite{debonis} and Bougherara et al. in \cite{boughP} studied parabolic equation involving $p$-Laplacian and singular nonlinearity, almost simultaneously. In \cite{badra}, authors studied problem $(P_t)$ with $p=q$ and the term $f(x,u)$ exhibiting the sub-homogeneous growth w.r.t. $p$ in the variable $u$. In this work, among other results, authors proved the existence of unique solution to $(P_t)$ when $\de<2+1/(p-1)$ by semi-discretization in time and using implicit Euler method. Furthermore, they proved the stabilization result under some additional assumptions on $f$ and the initial condition $u_0$. While authors in \cite{debonis} considered the nonlinear term in $(P_t)$ as $f(x,t)(u^{-\de}+1)$ with $p=q$, $\de>0$ and $0\le f\in L^r(0,T;L^m(\Om))$ where $\frac{1}{r}+\frac{N}{pm}<1$. Here, authors proved the existence of a positive solution. Subsequently, authors in \cite{boughP} studied problem $(P_t)$ when the function $f$ has a form of $f(x,u,\na u)$ and it exhibit super-homogeneous growth w.r.t. $p$ in $u$ variable. Assuming $u_0\in L^r(\Om)$, for some large $r$, authors proved the existence of a solution $u$ in $C^{0,\al}_{loc}(Q_T)$ such that $u^{(r-2+p)/p}\in L^p(0,T; W^{1,p}_0(\Om))$, for all $\de>0$, as the limit of solution to some auxiliary problem.\par
 For the case when the differential operator in problem $(P_t)$ is non-homogeneous in nature, we mention the contributions of B\"ogelein et al. in \cite{bogel}, Cai and Zhou in \cite{cai}, Baroni and Lindfors in  \cite{baroni}, and references therein. In \cite{baroni},
  authors considered the following kind of general quasilinear parabolic problem: 
  \begin{align}\label{OrlPr}
    u_t-\mathrm{div} A(\na u)=0 \ \mbox{ on }\Om_T; \quad u=\psi \ \mbox{ on }\pa_p\Om_T,
  \end{align}
 where $A:\mb R^N\to \mb R^N$ is a $C^1$ vector field exhibiting Orlicz type growth conditions
 and $\psi$ is a continuous function on the parabolic boundary $\pa_p\Om_T$. Here, authors proved the existence of a solution to \eqref{OrlPr}, which is continuous up to the boundary and local boundedness of its gradient. While in \cite{bogel,cai}, authors have proved the existence of solution for similar equations as in \eqref{OrlPr} with differential operator exhibiting Orlicz type growth.
 Diening et al. in \cite{dien} discussed the local regularity results for solution of problem of the type \eqref{OrlPr}. The study of parabolic problems involving nonhomogeneous operator and singular nonlinearity was still open after these works.\par
 
 Coming back to our paper, here we obtain the existence, uniqueness, regularity results and  asymptotic behaviour of the solution $u$ to problem $(P_t)$ depending on the growth of the nonlinear term $f$.  For the existence part in the sub-homogeneous case, we employ the technique of time discretization and implicit Euler method together with existence and regularity results of associated stationary problems (see $(S_\la)$ in Sec. 2). This also generalizes the work of \cite{badra} to nonhomogeneous differential operators case. 
 The novelty of the paper is the study of equations with combined effect of  the singular nonlinear terms and the non-homogeneous nature of the leading differential operator.  One of the main contribtion  of the paper is the construction of appropriate sub and super solutions of stationary problems corresponding to problem $(P_t)$ which are comparable to the initial condition $u_0$. In case of homogeneous operators, e.g., the operator $-\De_p$ (the $p$-Laplacian), it is well known that  the scalar multiple of $\phi_1$, the first eigenfunction of $-\De_p$ with zero Dirichlet boundary condition, can be used to construct sub and super solutions.
 Since the operator is not homogeneous in our case, we introduce a parameter in front of the singular term in the problem of type $(PS)$ (Sec. 2) and study the behaviour of the solution as the parameter goes to $0$ or $\infty$ by using the recently developed regularity results for purely singular problems in \cite{JDS}. Precisely, in Propositions \ref{sup} and \ref{prop3}, we prove that the solution $u_\rho$ converges uniformly to $0$, as $\rho\ra 0$ and to $\infty$, as $\rho\ra\infty$. These results help us to construct small subsolutions and large super solutions, (see \eqref{eq78}, \eqref{eq79}) of stationary problems of kinds $(S_\la)$ and $(P)$, that are comparable with initial condition $u_0$.  
Once the sub and super solutions with aforementioned properties are obtained, the method of time discretization and implicit Euler method similar to \cite{badra} applies to our case too (with some modifications). To establish the uniqueness result, we prove a comparison principle as in Theorem \ref{cmp} using D\'az-S\'aa inequality and consequently, using the local Lipschitz nature of $f$, we have the uniqueness of weak solution of $(P_t)$. Furthermore, using semigroup theory, we prove the stabilization result under some assumptions on the initial data, that is, the solution $u(t)$ converges to $u_\infty$, the solution of steady state problem $(P)$, in $L^\infty(\Om)$ as $t\ra\infty$.\par
At the end of Section 3, we prove the higher Sobolev integrability for the  solution of purely singular problem $(PS)$, as in Theorem \ref{sob-int}, which is of independent interest. Precisely, we prove the (unique) solution  to problem $(PS)$ belongs to $W^{1,m}_0(\Om)$ for $m<\frac{p-1+\de}{\de-1}$ for all $\de>1$ in the range $(p-2)\de<2(p-1)$. That is, for the case $1<p\le 2$, the result is true for all $\de>1$ and for $p>2$, it is true for $\de<2+2/(p-2)$. The significance of the result can be understood as follows.
In \cite{JDS}, it is proved that the solution $u$ of problem $(PS)$ is in $C^{1,\al}(\ov\Om)$, for some $\al\in(0,1)$ and thus in all $W^{1,m}_0(\Om)$ when $0<\de<1$. For the case $\de\ge 1$, $u$ is only $C^{0,\al}(\ov\Om)$ regular with $u\in W^{1,p}_0(\Om)$ if and only if $\de<2+1/(p-1)$ (see \cite[Theorems 1.4,1.7,1.8]{JDS}). Thus, the above mentioned regularity result of Theorem \ref{sob-int} improves the integrability condition of the gradient as $u\in W^{1,m}_0(\Om)$ for $m<(p-1+\de)/(\de-1)$, clearly which is bigger than $p$ in the case of $\de<2+1/(p-1)$, whereas for the case $2+1/(p-1)\leq \de <2+2/(p-2)$, it wasn't known in the previous works whether $u$ is in some Sobolev space. To prove the Theorem, using the approach of DiBenedetto and Manfredi \cite{dibMan} (see also \cite{caff}) we first prove higher integrability result for solution of some nonhomogeneous equation (involving a more general quasilinear operator) whose right hand side is the divergence of some vector field (see \eqref{eq81}). Then, we suitably choose this function satisfying the equation $-\Delta w= u^{-\de}+b(x)$, for $b\in L^\infty(\Om)$, with the help of Green function and thus complete the proof.  Subsequently, we use the higher Sobolev integrability result of Theorem \ref{sob-int} to improve the regularity of solution to $(P_t)$ as in corollary \ref{cor3}. Precisely, under some additional assumption on $u_0$, we prove that the solution $u\in C([0,T];W^{1,m}_0(\Om))$ for all $p\le m<(p-1+\de)/(\de-1)$. To the best of our knowledge, these regularity results are new even for the equations involving $p$-Laplacian operator with singular nonlinearities.\par
Concerning the case of super-homogeneous growth, we apply nonlinear semigroup theory and Picard iteration process to obtain first a solution to the problem $(P_t)$ where the nonlinear term is replaced by $f(v)$, for $v\in L^\infty(Q_T)$. Then, applying Banach fixed point theorem, we get the local existence result as in Theorem \ref{thm7}. For the blow up result, the available literature deals with only the equations involving only superhomogeneous nonlinearity. In this work, we generalise the previous works of blow up phenomena to the equations where the nonlinear term involves singularity. 
The main difficulty in this regard is presence of the singular term in the equation, which causes non differentiability of the energy functional. We overcome this difficulty by establishing a suitable comparison of the solution to a sufficiently regular function (see lemma \ref{lem9}). Additional difficulty is that, for the case $\de<1$, the singular nonlinearity makes the energy functional to exhibit concave-convex type growth which forces us to prove the blow up behaviour only for a small parameter. Using the energy method and concavity method, we prove that the solution $u$ of problem $(P_t)$ blows up in the sense that $\| u(t)\|_{L^\infty(\Om)}\ra\infty$ as $t\ra\infty$, for a small parameter $\vth>0$ (its threshold value depends only on the initial data) in the case $\de\leq 1$ while for all $\vth>0$ in the case $\de>1$.   To the best of our knowledge, there is no work dealing with the blow up phenomenon for equation involving singular as well as super-homogeneous nonlinearity even for the case $p=q$ (that is equations involving only homogeneous operator $-\De_p$).

%
 
 \section{Main results}
 In this section, we state our main results concerning the existence, uniqueness, regularity and asymptotic behaviour of weak solutions of problem $(P_t)$. Here we also discuss several results regarding the stationary problems which are used to prove the aforementioned properties of solution to $(P_t)$. Based on the growth rate of the nonlinear term $f$, we distinguish the two cases as sub-homogeneous problem and super-homogeneous problem. \\
 For the sub-homogeneous growth condition, we assume that the function $f$ satisfies the following:
 \begin{enumerate}
  \item[$(f_1)$] $0\le\lim_{s\ra\infty}\frac{f(x,s)}{s^{q-1}}:=\al_f<\infty$, and
  \item[$(f_2)$] the map $s\mapsto\frac{f(x,s)}{s^{q-1}}$ is nonincreasing in $\mb R^+$ for a.e. $x\in\Om$.
 \end{enumerate}
 We first study the following general form of problem $(P_t)$,
  \begin{equation*}
    (G_t)\; \left\{\begin{array}{rllll}
     u_t-\Delta_{p}u -\Delta_{q}u & =\vth \; u^{-\de}+ g(x,t), \; u>0 \text{ in } Q_T, \\ u&=0 \quad \text{ on }  \Ga_T, \\
     u(x,0)&= u_0(x) \; \text{ in }\Om,
    \end{array}
    \right.
  \end{equation*}
 where $g\in L^\infty(Q_T)$. 
 \begin{Definition}
  Let $\mc V(Q_T): = \big\{ u\in L^\infty(Q_T)\; : \: u_t\in L^2(Q_T) \mbox{ and } u\in L^\infty\big(0,T; W^{1,p}_0(\Om)\big)  \big\}$ and 
  for $u\in\mc V(Q_T)$, we say that $u>0$ in $Q_T$ if for any compact set $K\subset Q_T$, $\textnormal{ess}\inf_K u>0$.
 \end{Definition}
 We define the notion of weak solution to problem $(G_t)$ as follows.
 \begin{Definition}
 A function $u\in\mc V(Q_T)$ is said to be a weak solution of $(G_t)$ if $u>0$ in $Q_T$, $u(x,0)=u_0(x)$ a.e. in $\Om$ and it satisfies 
  \begin{align*}
  	\int_{Q_T} \left( \phi\frac{\pa u}{\pa t}+\big(|\na u|^{p-2}+ |\na u|^{q-2} \big)\na u\na\phi -\big(\vth \; u^{-\de}+ g(x,t) \big)\phi \right)dx dt=0,
  \end{align*}
  for all $\phi\in\mc V(Q_T)$.
 \end{Definition}
\begin{Definition}
 We define the conical shell $\mc C_\de\subset L^\infty(\Om)$ as the set of functions $u$ satisfying the following in $\Om$, 
 \[ c_1 \varphi_\de(d(x)) \leq u(x) \leq c_2 \varphi_\de(d(x)), \]
  where $d(x):= \textnormal{dist}(x,\pa\Om)$, $c_1,c_2>0$ are constants and for $A>0$ large enough, the function $\varphi_\de:[0,\infty)\to [0,\infty)$ is defined as
 \begin{align*}
	\varphi_\de(s):=\begin{cases}
	s\qquad \mbox{if }\de<1,\\
	s\log^{1/p}\Big(\frac{A}{s}\Big) \quad \mbox{if }\de=1, \\
	s^\frac{p}{p-1+\de} \quad \mbox{if }\de>1.
	\end{cases}
	\end{align*}
 \end{Definition}
 We follow the approach similar to \cite{badra}, to obtain the existence result for $(G_t)$. In this regard, we consider the following stationary problem:  for $h\in L^\infty(\Om)$ and $\la>0$,
  \begin{equation*}
    (S_\la)\; \left\{\begin{array}{rllll}
      u-\la\big(\Delta_{p}u +\Delta_{q}u+ \vth \; u^{-\de}\big) & = h, \; u>0 \text{ in } \Om, \\ u&=0 \quad \text{ on }  \pa\Om.
     \end{array}
    \right.
  \end{equation*}
  \begin{Theorem}\label{thm1}
    Let $h\in L^\infty(\Om)$ and $0<\de<2+1/(p-1)$. Then, for any $\la,\vth>0$, problem $(S_\la)$ admits a unique solution $u_\la\in W^{1,p}_0(\Om)\cap\mc C_{\de}\cap C_0(\ov\Om)$.
 \end{Theorem}
\begin{Definition}
  For the case $\de\ge 2+1/(p-1)$,  $u\in W^{1,p}_{loc}(\Om)$ is said to be a weak solution of problem $(S_\la)$ if 
	\begin{align*}
	\int_{\Om} \big(u\phi+\la|\na u|^{p-2}\na u\na\phi+ \la|\na u|^{q-2}\na u\na\phi\big)dx =\int_{\Om} \big(\la \vth \;u^{-\de}+ h \big)\phi~dx,
	\end{align*}
  for all $\phi\in C_c^\infty(\Om)$ and the Dirichlet datum is understood in the sense that there exists $\nu\ge 1$ such that $u^\nu\in W^{1,p}_0(\Om)$. 
\end{Definition}
\begin{Theorem}\label{thm3}
Let $\de\ge 2+1/(p-1)$ and $h\in L^\infty(\Om)$. Then, for any $\la,\vth>0$, there exists a weak solution $u\in W^{1,p}_{loc}(\Om)\cap\mc C_\de\cap C_0(\ov\Om)$ of problem $(S_\la)$ such that $u\not\in W^{1,p}_0(\Om)$.	
\end{Theorem}
For the stabilization result, we study the following stationary problem associated to $(P_t)$,
\begin{equation*}
(P)\; \left\{\begin{array}{rllll}
-\Delta_{p}u -\Delta_{q}u & = \vth \; u^{-\de}+ f(x,u), \; u>0 \text{ in } \Om, \\ u&=0 \quad \text{ on }  \pa\Om. 
\end{array}
\right.
\end{equation*}
\begin{Theorem}\label{thm4}
	Let $0<\de<2+1/(p-1)$ and assume that (f1) and (f2) are satisfied. Then, for all $\vth>0$, there exists a unique solution $u_\infty$ of problem $(P)$ in $W^{1,p}_0(\Om)\cap\mc C_\de\cap C_0(\ov\Om)$.
\end{Theorem}
Now, we state our result concerning higher Sobolev integrability.
\begin{Theorem}\label{sob-int}
 Let $\de>1$ and $u$ be the solution to the following problem:
 \begin{equation*}
  (PS)\left\{\begin{array}{rllll}
  -\De_p u -\De_q u& = u^{-\de}+b(x),  \quad u>0 \quad \mbox{in }\Om, \\
  u&=0 \quad \mbox{on }\pa\Om,
  \end{array}
  \right. 
 \end{equation*}
 where $b\in L^\infty(\Om)$. If $\de$ satisfies $(p-2)\de < 2(p-1)$, then $ u\in W_0^{1,m}(\Om)$ for all $m<\frac{p-1+\de}{\de-1}$.
\end{Theorem}
\begin{Remark}
 We remark that the result of Theorem \ref{sob-int} holds for equations involving a more general class of operators with singular nonlinearity as in $(PS)$, for example
 \begin{enumerate}
 	\item[(i)] The operator $-\textnormal{div}\big(|\na u|^{p-2}\na u + a(x) |\na u|^{q-2} \na u  \big)$, where $0\le a(x)\in W^{1,\infty}(\Om)\cap C(\ov\Om)$ with $1<q<p<\infty$.
 	\item[(ii)] The operator $-\textnormal{div}\big(a(x)|\na u|^{p-2}\na u +  |\na u|^{q-2} \na u  \big)$, where $0<\inf_{\bar\Om} a(x)\le a(x)\in  C^1(\ov\Om)$ with $1<q<p$.
 \end{enumerate}
\end{Remark}
Concerning the parabolic case, our first existence theorem in this regards is stated below.
\begin{Theorem}\label{thm2}
 Let $0<\de<2+1/(p-1)$, $g\in L^\infty(Q_T)$ and $u_0\in W^{1,p}_0(\Om)\cap\mc C_\de$. Then, for all $\vth>0$, there exists a unique solution $u$ to problem $(G_t)$ such that $u(t)\in\mc C_\de$ uniformly for $t\in[0,T]$. Moreover, $u\in C([0,T];W^{1,p}_0(\Om))$ and the following holds, for $\de\neq 1$, and for all $t\in[0,T]$,
	\begin{equation}\label{eq7}
	\begin{aligned}
	&\int_{0}^{t}\int_{\Om} \Big(\frac{\pa u}{\pa t}\Big)^2 dxd\tau+\frac{1}{p} \int_{\Om} |\na u|^p dx+ \frac{1}{q} \int_{\Om} |\na u|^q dx- \frac{\vth}{1-\de} \int_{\Om} u^{1-\de}dx \\
	&= \int_{0}^{t}\int_{\Om} g\frac{\pa u}{\pa t} dxd\tau+\frac{1}{p} \int_{\Om} |\na u_0|^p dx+ \frac{1}{q} \int_{\Om} |\na u_0|^q dx- \frac{\vth}{1-\de} \int_{\Om} u_0^{1-\de}dx.
	\end{aligned}
	\end{equation}
 (For $\de=1$, terms of the form $v^{1-\de}/(1-\de)$ are replaced by $\log v$ in the above expression).
\end{Theorem}
 
For problem $(P_t)$, we have existence result as follows.
\begin{Theorem}\label{thm5}
 Let $0<\de<2+1/(p-1)$ and $u_0\in W^{1,p}_0(\Om)\cap\mc C_\de$. Suppose (f1) and (f2) are satisfied. Then, for any $T>0$ and $\vth>0$, there exists a unique weak solution $u$ of $(P_t)$ such that $u(t)\in\mc C_\de$ uniformly for $t\in[0,T]$, $u\in C\big([0,T]; W^{1,p}_0(\Om)\big)$ and the following holds for $\de\neq 1$ and for any $t\in[0,T]$,
 \begin{equation*}
 \begin{aligned}
  &\int_{0}^{t}\int_{\Om} \Big(\frac{\pa u}{\pa t}\Big)^2 dxd\tau+\frac{1}{p} \int_{\Om} |\na u|^p dx+ \frac{1}{q} \int_{\Om} |\na u|^q dx- \frac{\vth}{1-\de} \int_{\Om} u^{1-\de}dx \\
  &= \int_{\Om} F(x,u(t)) dx+\frac{1}{p} \int_{\Om} |\na u_0|^p dx+ \frac{1}{q} \int_{\Om} |\na u_0|^q dx- \frac{\vth}{1-\de} \int_{\Om} u_0^{1-\de}dx -\int_{\Om} F(x,u_0)dx,
 \end{aligned}
 \end{equation*}
 where $F(x,w):=\int_{0}^{w} f(x,s)ds$. (For $\de=1$, the terms of the form $v^{1-\de}/(1-\de)$ are replaced by $\log v$ in the above expression).
\end{Theorem}
\begin{Remark}
  For the case $\de\ge 2+1/(p-1)$, using Theorem \ref{thm3} together with the comparison principle of \cite[Theorem 1.5]{JDS} and the uniform behaviour of solution to $(S_\la)$ with respect to the distance function, one can generalize the existence result of Theorem \ref{thm5} for problem $(P_t)$ such that the solution $u^{(p-1+\nu)/p}\in L^p(0,T;W^{1,p}_0(\Om))$, for some $\nu\ge 2$, whenever $u_0\in W^{1,p}_0(\Om)\cap\mc C_\de$.
\end{Remark}
Furthermore, we prove the following regularity result for solution of problem $(P_t)$. Set 
\[ \mc D(\mc A):= \big\{ v\in\mc C_\de\cap W^{1,p}_0(\Om) : \mc A v:= -\De_p v-\De_q v- \vth \; v^{-\de}\in L^\infty(\Om) \big\}.  \]
\begin{Proposition}\label{prop2}
  Assume the hypotheses of Theorem \ref{thm5} are satisfied and $u_0\in\ov{\mc D(\mc A)}^{L^\infty(\Om)}$. Then the solution $u$ of problem $(P_t)$ belongs to $C([0,T]; C_0(\ov\Om))$ and 
	\begin{enumerate}
	  \item[(i)] If $v$ is a solution of $(P_t)$ with initial datum $v_0\in\ov{\mc D(\mc A)}^{L^\infty(\Om)}$, then the following holds
		\begin{align*}
		\|u(t)-v(t)\|_{L^\infty(\Om)} \le e^{\omega t}\|u_0-v_0\|_{L^\infty(\Om)}, \quad 0\le t\le T.
		\end{align*}
	  \item[(ii)] If $u_0\in\mc D(\mc A)$, then $u\in W^{1,\infty}(0,T; L^\infty(\Om))$ with $\De_p u+\De_q u+\vth \; u^{-\de}\in L^\infty(Q_T)$ and
		\begin{align*}
		\bigg\|\frac{du(t)}{dt}\bigg\|_{L^\infty(\Om)}\le e^{\omega t} \|\De_p u_0+\De_q u_0 +\vth u_0^{-\de}+f(x,u_0)\|_{L^\infty(\Om)},
		\end{align*}
	 \end{enumerate} 
  where $\omega>0$ is the Lipschitz constant for $f$ in $[\ul u,\ov u]$ with $\ul u$ and $\ov u$ as sub and supersolution to $(P)$ as constructed in the proof of Theorem \ref{thm4}, respectively. 
\end{Proposition}
 \begin{Corollary}\label{cor3}
 	Let the hypotheses of Theorem \ref{thm5} be true and assume $u_0\in\mc D(\mc A)$. Then, the solution $u$ of problem $(P_t)$ belongs to $C([0,T]; W^{1,m}_0(\Om))$ for all $m<\frac{p-1+\de}{\de-1}$, in the case of $1<\de<2+1/(p-1)$.\\
 	(Note that the result is true for all $m>1$ in the case of $\de<1$ from the $C^{1,\al}(\ov\Om)$ regularity, $\al\in(0,1)$, of solutions to $(PS)$.)
 \end{Corollary}
Next, we prove the following stabilization result.
\begin{Theorem}\label{thm6}
  Under the hypotheses of Theorem \ref{thm5}, the solution $u$ of problem $(P_t)$ is defined in $\Om\times(0,\infty)$ and it satisfies
  \begin{align*}
  	u(t)\ra u_\infty \quad\mbox{in }L^\infty(\Om) \ \mbox{ as }t\ra\infty,
  \end{align*}
  where $u_\infty$ is the solution to stationary problem $(P)$.
\end{Theorem}
Now, we turn to the super-homogeneous case. That is, we assume $f(x,u)$ is of the form $f(u)$ only and $f\in C(\mb R)$ together with
\begin{enumerate}
 \item[(f3)] there exists $r\in (q,p^*)$ and $c_r>0$ such that $c_r |s|^r\leq r F(s)\leq sf(s)$ for all $s\in\mb R$, where $F(u)=\int_{0}^{u}f(t)dt$, and $p^*:=Np/(N-p)$ if $p<N$ and $p^*<\infty$ be arbitrary large number if $p\geq N$.
\end{enumerate}
 We have the local existence result as below.
\begin{Theorem}\label{thm7}
 Let $u_0\in W^{1,p}_0(\Om)\cap\mc C_\de$.	
 Then, for all $\vth>0$ and $\de<2+1/(p-1)$, there exists $\hat T>0$ small enough, such that the problem $(P_t)$ has a solution $u\in C([0,T];L^\infty(\Om))$, which is defined in $Q_T$ for all $T<\hat T$.
\end{Theorem}
\begin{Remark}
 We remark that the solution $u$ obtained in Theorem \ref{thm7} belongs to the space $C([0,T];W^{1,p}_0(\Om))$ and the energy estimate of Theorem \ref{thm5} holds in this case too. Moreover, the results of Proposition \ref{prop2} and Corollary \ref{cor3} also hold for small $T>0$ (possibly less than $\hat T$). Proofs of the aforementioned results are essentially same as in the sub-homogeneous case by noticing the fact that the solution is in $C([0,T];L^\infty(\Om))$ and the local Lipschitz nature of $f$ makes the problem $(P_t)$ of similar kind to problem $(G_t)$.
 \\ Moreover, as a consequence of the integral representation of the solution as in \eqref{eq94}, we see that the solution exists globally or there exists $T^*<\infty$ such that $\| u \|_{L^\infty(\Om)}\ra\infty$ as $t\ra T^*$.
\end{Remark}
 To obtain the blow up result, for $u\in W^{1,p}_0(\Om)$, we define the following, 
 \begin{align*}
 &J_\vth(u):= \frac{1}{p}\int_{\Om} |\na u|^p + \frac{1}{q}\int_{\Om} |\na u|^q -\frac{\vth}{1-\de} \int_{\Om} |u|^{1-\de}-\int_{\Om} F(u) \\
 &I_\vth(u):= \int_{\Om} |\na u|^p + \int_{\Om} |\na u|^q - \vth \int_{\Om} |u|^{1-\de}-\int_{\Om} f(u)u.
 \end{align*}	
 For $\de=1$, the term $\frac{1}{1-\de}|u|^{1-\de}$ in the first equation is replaced by $\log |u|$.
 The corresponding Nehari manifold is defined as 
 \begin{align*}
 \mc N_{\vth}:= \{ u\in W^{1,p}_0(\Om) \ : \ I_{\vth}(u)=0\} \quad \mbox{ with } \  \Theta_{\vth}=\inf_{u\in\mc N_{\vth}}J_{\vth}(u).
 \end{align*}
 We note that $\Theta_{\vth}<0$, for the case $\de<1$. Indeed, proceeding similar to \cite[Lemma 4.4]{dpk}, we can prove that the infimum over $\mc N^+ \cup\mc N^0\subset\mc N$ is negative for $\vth<\la_*$. 	
\begin{Theorem}\label{thm8}
  Let the hypotheses in Theorem \ref{thm7} be satisfied and the function $f$ satisfies (f3) with $r\in [p,p^*)$ and $r> 2$. Then, the solution $u$ of problem $(P_t)$ blows up in finite time in the sense that, there exists $T^*<\infty$ such that
 	\begin{align*}
 	 \lim_{t\ra T^*} \| u(\cdot,t) \|_{L^\infty(\Om)}=\infty,
 	\end{align*}
 under the following conditions:
  \begin{enumerate}
    \item [(i)] For $\de \leq 1$, $2N/(N-2)\le q$ and for all $\vth<\vth_*$, where $\vth_*>0$ is a constant, if $J_\vth(u_0)\leq \Theta_{\vth}$ for $\de<1$ and $J_\vth(u_0)\leq \min\{\Theta_{\vth},0\}$ for $\de=1$, and $I_\vth(u_0)<0$. 
    \item[(ii)] For $\de\in \big(1, 2+\frac{1}{p-1}\big)$ and for all $\vth>0$, if $J_\vth(u_0)\leq 0$.
  \end{enumerate}
 \end{Theorem}

\section{Existence and regularity results for stationary problems}
 In this section, we prove the existence and uniqueness of solution to stationary problems associated to $(P_t)$ using the method of sub and super solution. Furthermore, using Diaz-S\'aa inequality, we obtain the comparison principle as in Theorem \ref{cmp}. Moreover, we prove higher Sobolev integrability of solution to the purely singular problem, see Theorem \ref{sob-int}.\\
 To construct a suitable subsolution for the case $\de<1$, we recall the following proposition proved by Papageorgiou et al. \cite{papageorg}. 
  For this purpose, we define the following set
 \[ \text{int }C_+:= \big\{ u\in C^1(\ov\Om): u>0 \mbox{ in }\Om, u=0 \mbox{ on }\pa\Om, \frac{\pa u}{\pa\nu} \Big|_{\pa\Om}<0 \big\}.\]
 \begin{Lemma}\cite[Proposition 10]{papageorg} \label{lemsub}
  For all $\rho>0$, there exists a unique solution $\tl u_\rho\in\text{int } C_+$ to the following problem
 	\begin{align}\label{eqS}
 	-\De_p u -\De_q u =\rho \quad \mbox{in }\Om, \quad u=0 \ \ \mbox{on }\pa\Om.
 	\end{align}
  Furthermore, the map $\rho\mapsto \tl u_\rho$ is increasing from $(0,1]$ to $C^1_0(\ov\Om)$ and $\tl u_\rho\ra 0$ in $C^1_0(\ov\Om)$ as $\rho\ra 0^+$.
 \end{Lemma}
\begin{Corollary}\label{cor1}
	We have $\frac{\tl u_\rho}{d}\ra 0$ uniformly in $\Om$, as $\rho\ra 0$. 
\end{Corollary}
Below we mention results which will be used in the sequel to construct supersolutions.
\begin{Proposition}\label{sup}
 Let $M>1$, $L,l\ge 0$ and $\de<1$. Then,
 there exists $v_M\in W^{1,p}_0(\Om)\cap\mc C_\de$, solution to the following problem: 
 	 \begin{equation}\label{eq52}
 	  (\mc M)\; \left\{\begin{array}{rllll}
 	  -\Delta_{p}u -\Delta_{q}u & = M u^{-\de}+l u^{q-1}+L, \; u>0 \text{ in } \Om, \\ u&=0 \quad \text{ on }  \pa\Om.
 	  \end{array}
 	 \right.
 	 \end{equation}
 Moreover, $\frac{v_M}{d}\ra \infty$  uniformly in $\Om$, as $M\ra\infty$. 
\end{Proposition}
\begin{proof}
We first consider the following problem:
 \begin{align*}
 -\De_p v -\De_q v = v^{-\de} \quad \mbox{in }\Om, \quad
 v=0 \quad \mbox{on }\pa\Om.
 \end{align*}
 Then, by \cite[Theorem 1.4]{JDS} with $\beta=0$, there exists a unique positive solution $\ul v\in W^{1,p}_0(\Om)\cap\mc C_\de$. Next, set 
 \begin{align*}
 \tl h(x,v(x)):= \begin{cases}
 v(x)^{-\de} \quad \mbox{if }v(x)\ge \ul v(x) \\
 \ul v(x)^{-\de} \quad \mbox{otherwise}
 \end{cases}
 \end{align*}
 and $\tl H(x,t):=\int_{0}^{t} \tl h(x,s)ds$.
 Define $E: W^{1,p}_0(\Om) \to\mb R$ as 
 \begin{align*}
 E(u)= \frac{1}{p} \int_{\Om} |\na u|^p dx+\frac{1}{q} \int_{\Om} \big(|\na u|^q- l|u|^q\big) dx- M\int_{\Om} \tl H(x,u) dx-L\int_{\Om} u dx.
 \end{align*}
 Due to the fact that  $l \ge 0$, $\de<1$ and $q<p$, $E$ is coercive and weakly lower semi continuous in $W^{1,p}_0(\Om)$, and hence bounded below. Let $\{v_n\}\in W^{1,p}_0(\Om)$ be a minimizing sequence for $E$. Since $\inf_{u\in W^{1,p}_0(\Om)} E(u)<0$, it is easy to prove that the sequence $\{v_n\}$ is bounded in $W^{1,p}_0(\Om)$. Therefore, up to a subsequence, we have the following
 \begin{align*}
 v_n \rightharpoonup v_M \quad \mbox{ weakly in } W^{1,p}_0(\Om) \quad \mbox{and } v_n \ra v_M \quad \mbox{in }L^m(\Om), \mbox{ for }1\le m<p^*,
 \end{align*} 
 where $p^*:= Np/(N-p)$ for $p<N$ and $p^*<\infty$ otherwise. Using weak lower semicontinuity of norms and the aforementioned compactness results, we get that $v_M$ is a global minimizer for $E$ in $W^{1,p}_0(\Om)$. By \cite[Lemma A2]{giacomoni}, we know that $E$ is G\^ ateaux differentiable and thus $v_M$ satisfies 
 \begin{equation*}
 \left\{\begin{array}{rllll}
  -\De_p v_M -\De_q v_M = M \tl h(x,v_M) +l v_M^{q-1} +L \ \ \mbox{ in }\Om; \quad
  v_M=0 \ \mbox{ on }\pa\Om.
 \end{array}
 \right. 
 \end{equation*} 
 Employing comparison principle, we obtain $v_M\ge \ul v$, therefore $v_M$ is a weak solution of $(\mc M)$. Next, since $\de<1$, from \cite[Lemma 3.2]{dpk}, we obtain $v_M\in L^\infty(\Om)$. Consequently, using \cite[Remark 2.8]{JDS}, we get that $v_M\in\mc C_\de$.
 For the second part, we divide the equation in $(\mc M)$ by $M$ and substitute $w_M=M^{-1/(p-1+\de)}v_M$, then 
 \begin{equation}\label{eq50}
  \left\{\begin{array}{rllll}
  -\De_p w_M -M^\frac{q-p}{p-1+\de}\De_q w_M& = w_M^{-\de} +l \; M^\frac{q-p}{p-1+\de} w_M^{q-1} +L \;M^\frac{1-p}{p-1+\de} \quad \mbox{in }\Om, \\
  w_M&=0 \quad \mbox{on }\pa\Om.
  \end{array}
  \right. 
 \end{equation} 
 Due to the fact $M>1$, we note that the coefficients of terms on the right side of \eqref{eq50} are bounded above by quantities which are independent of $M$. Therefore, careful reading of \cite[Theorem 2, P. 361]{He} and \cite[Lemma 3.2]{dpk} implies that $\|w_M\|_{ L^\infty(\Om)}\leq C_1$, where $C_1>0$ is a constant independent of $M$. Now, using Remark 2.8 and Proposition 2.7 of \cite{JDS}, we get $w_M\leq C_2 d$ in $\Om$, where $C_2>0$ is a constant independent of $M$. Thus, invoking \cite[Theorem 1.7]{JDS}, we obtain $\| w_M \|_{C^{1,\ga}(\ov\Om)}\le C_3$, where $\ga\in(0,1)$ and $C_3>0$ is a constant independent of $M$. Therefore, there exists $w\in C^1(\ov\Om)$ such that $w_M\ra w$ as $M\ra\infty$. By uniqueness of solution (see \cite[Lemma 3.1]{giacomoni}) to the problem
 \begin{equation*}
 	\left\{\quad-\De_p u = u^{-\de},\ \ u>0 \quad\mbox{in }\Om; \qquad u=0 \quad \mbox{on }\pa\Om,\right.
 \end{equation*}
 we get that $w\not\equiv 0$ in $\Om$ and it satisfies the above equation. Due to the convergence of $w_M$ to $w$ in $C^1(\ov\Om)$, we conclude that 
 \[ \frac{w_M}{d} \ra \frac{w}{d} \quad\mbox{ uniformly in }\Om, \ \mbox{ as }M\ra\infty. \] 
 Since $w> 0$ in $\Om$, we get the required result of the proposition by using definition of $w_M$. \QED
\end{proof}
Next, we consider the case $\de \ge 1$.
In what follows, we fix the constants $c_1,c_2$ appearing in the definition of $\mc C_\de$ for $u_0$ as $k_1$ and $k_2$, respectively. That is, for the case $\de>1$,
\begin{align}\label{init}
 k_1 d(x)^{p/(p-1+\de)} \le u_0(x) \le k_2 d(x)^{p/(p-1+\de)} \quad\mbox{in }\Om.	
\end{align} 
\begin{Proposition}\label{prop3}
 Let $M\ge 1$, $L\geq 0$, $\rho>0$ and $1\leq\de<2+1/(p-1)$. Then,
	\begin{enumerate}
	  \item[(a)] There exists $v_\rho\in W^{1,p}_0(\Om)\cap\mc C_\de$, solution of  the following equation:       
		\begin{equation}\label{eq72}
		\left\{\begin{array}{rllll}
		-\De_p v -\De_q v& =\rho v^{-\de},  \quad v>0 \quad \mbox{in }\Om, \\
		v&=0 \quad \mbox{on }\pa\Om,
		\end{array}
		\right. 
		\end{equation}
	   such that the sequence $v_\rho \ra 0$ uniformly in  $K$, as $\rho\ra 0$, for every compact subset $K$ of $\Om$. Moreover, for sufficiently small $\rho>0$, the following holds
	  \begin{align}\label{eq78}
	  	v_\rho(x) \leq \begin{cases}
	  	  k_1 d(x)^\frac{p}{p-1+\de}  \ \mbox{ if }\de>1\\
	  	  k_1 d(x)\log(A/d(x)) \ \mbox{ if }\de=1
	  	\end{cases}\quad\mbox{in }\Om,
	  \end{align}
	  where $k_1$ is as in \eqref{init} and $A$ is large enough constant.
	  \item[(b)]  Consider the problem:  
	  \begin{equation}\label{eq71}
	  \left\{\begin{array}{rllll}
	  -\Delta_{p}u -\Delta_{q}u & = M u^{-\de}, \; u>0 \text{ in } \Om, \\ u&=0 \quad \text{ on }  \pa\Om.
	  \end{array}
	  \right.
	  \end{equation}
 Then, there exists $u_M\in W^{1,p}_0(\Om)\cap\mc C_\de$, solution of \eqref{eq71} such that the sequence $u_M\ra \infty$ uniformly in $K$, as $M\ra\infty$, for every compact subset $K$ of $\Om$. Furthermore, for sufficiently large $M$, the following holds
 \begin{align}\label{eq79}
  u_M(x) \geq \begin{cases}
  k_2 d(x)^\frac{p}{p-1+\de}  \ \mbox{ if }\de>1\\
  k_2 d(x)\log(A/d(x)) \ \mbox{ if }\de=1
  \end{cases}\quad\mbox{in }\Om,
 \end{align}
 where the constant $k_2$ appears in \eqref{init} and $A$ is large enough.
\end{enumerate}
\end{Proposition}
\begin{proof} 
\textbf{Proof of (a)}:\\
 Proceeding similar to the proof of \cite[Theorem 1.4]{JDS} for $\ba=0$, we get the existence of $v_\rho\in W^{1,p}_{loc}(\Om)\cap\mc C_\de$. Moreover, since $\de<2+1/(p-1)$, we have $v_\rho\in W^{1,p}_0(\Om)$ and the regularity result of \cite[Theorem 1.8]{JDS} implies $v_\rho\in C^{0,\al}(\ov\Om)$, for some $\al\in(0,1)$ (which may depend on $\rho$). 
 By weak comparison principle, we have $ v_\rho \leq v_1$ for all $\rho\in (0,1)$. Since $v_1\in \mc C_\de$, there exists $c_1>0$ such that 
\begin{align*}
0< v_\rho(x) \le v_1(x) \leq c_1 d(x)^{p/(p-1+\de)} \leq c_1 \mathrm{diam}(\Om)^{p/(p-1+\de)} \quad\mbox{in }\Om.
\end{align*} 
That is, $v_\rho\leq c_1 d(x)^{p/(p-1+\de)}$ and $\| v_\rho\|_{L^\infty(\Om)} \leq c_2$, where $c_2$ is independent of $\rho$. Therefore, from the proof of claim (i) of \cite[Theorem 1.8]{JDS}, we see that 
\begin{align*}
\int_{B_r} |\na v_\rho|^p \leq C_3 \; r^{N-p},
\end{align*}
here $C_3>0$ is a constant independent of $\rho$. Thus, completing the proof similarly as in the Theorem, we get that $v_\rho\in C^{0,\al}(\ov\Om)$ and the bound on the norm is independent of $\rho$. Therefore, by Arzela-Ascoli theorem, up to a subsequence, $v_\rho\ra v$ uniformly in $K$, for every compact subset $K$ of $\Om$. Moreover, from the equation \eqref{eq72}, it is clear that $v=0$ in $\Om$, because
the equation 
\begin{align*}
-\De_p v -\De_q v =0 \; \mbox{ in }\Om \quad v=0 \; \mbox{ on }\pa\Om,
\end{align*}  
has only trivial solution. Therefore, $v_\rho \ra 0$ uniformly in  $K$, as $\rho\ra 0$, for every compact subset $K$ of $\Om$.\\
Now, we will show that $v_\rho(x) \leq k_1 d(x)^{p/(p-1+\de)}$, for sufficiently small $\rho>0$. Set $\ul w = k_1 d(x)^\tau$ with $\tau={p/(p-1+\de)}$. Then, it is easy to observe that 
{\small \begin{align*}
-\De_p \ul w - \De_q \ul w = \frac{k_1^\de}{\ul w^\de}\Big[(k_1\tau)^{p-1}\big(-\De d \; d+(p-1)(1-\tau)\big)+ (k_1\tau)^{q-1} \big(-\De d \; d+(q-1)(1-\tau)\big) \Big] 
\end{align*}}
Furthermore, noticing the fact that the boundary $\pa\Om$ is of class $C^2$, therefore $d\in C^2( \Om_\mu)$ and $|\De d|\le D$ in $\Om_\mu$. We choose $0<\varrho\leq \frac{(q-1)(1-\tau)}{2 D}$, which implies that 
\begin{align}\label{eq73}
-\De_p \ul w - \De_q \ul w \geq \frac{k_1^\de}{\ul w^\de} \min\{ (k_1\tau)^{p-1}, (k_1\tau)^{q-1}\}\frac{(q-1)(1-\tau)}{2}:= \frac{c(k_1)}{\ul w^\de}\quad \mbox{in }\Om_\varrho. 
\end{align}   
Since  $v_\rho \ra 0$ uniformly in $K$, for every compact subset $K$ of $\Om$, there exists $\rho_1>0$ such that $\rho_1 \leq c(k_1)$ and for all $\rho\le \rho_1$,
\begin{align*}
v_\rho \leq k_1 d(x)^\tau \quad\mbox{ in }\Om\setminus\Om_\varrho.
\end{align*}
Therefore, taking into account \eqref{eq72} and \eqref{eq73}, for all $\rho\le \rho_1<c(k_1)$, we have 
\begin{align*}
-\De_p \ul w - \De_q \ul w - c(k_1) \ul w^{-\de} \geq -\Delta_{p}v_\rho -\Delta_{q}v_\rho- c(k_1) v_\rho^{-\de},
\end{align*} 
which, by weak comparison principle, implies $v_\rho \leq k_1 d(x)^\tau$ in $\Om_\varrho$. Hence, for $\rho>0$ small enough, we have $v_\rho$ satisfies \eqref{eq72} and $v_\rho(x) \le k_1 d(x)^\tau \leq u_0(x)$ in $\Om$.\par
\textbf{Proof of (b)}:\\  
 Proceeding similar to the case (a), we get the existence of solution  $u_M\in W^{1,p}_0(\Om)\cap\mc C_\de$ of \eqref{eq71}. 
 Next, we prove the convergence result for sequence $u_M$ in the interior of $\Om$ in three steps.\\
 \textit{Step (i)}: Making the substitution $v_M=M^{-1/(p-1+\de)}u_M$, we see that \eqref{eq71} transform into
	\begin{equation}\label{eq74}
	\left\{\begin{array}{rllll}
	-\De_p v_M -M^\frac{q-p}{p-1+\de}\De_q v_M& = v_M^{-\de}   \quad \mbox{in }\Om; \quad
	v_M&=0 \ \mbox{ on }\pa\Om.
	\end{array}
	\right. 
	\end{equation} 
	Since $u_M\in W^{1,p}_0(\Om)\cap\mc C_\de$ for all $M\ge 1$ (here coefficients of the function $d^\tau$ may depend on $M$), we see that $u_M^\ga\in W^{1,p}_0(\Om)$ for all $\ga\geq 1$. Therefore, for fixed $\ga>1$ (to be specified later), we set $w_M = v_M^\ga$ and then \eqref{eq74} implies 
	\begin{align}\label{eq75}
	-\De_p w_M- &M^\frac{q-p}{p-1+\de} \ga^{p-q} w_M^{(\ga-1)(p-q)/\ga}\De_q w_M + (\ga-1)(p-1) \frac{|\na w_M|^p}{w_M} 
	\nonumber\\
	&\quad+ M^\frac{q-p}{p-1+\de}\ga^{p-q}\frac{(\ga-1)(q-1)|\na w_M|^q}{w_M^{1+(\ga-1)(q-p)/\ga}} 
	= \ga^{p-1} w_M^{\frac{-\de+(\ga-1)(p-1)}{\ga}}.
	\end{align}
	We choose $\ga>1$ such that ${-\de+(\ga-1)(p-1)}>0$, this implies that $\frac{-\de+(\ga-1)(p-1)}{\ga}\in(0,p-1)$. Therefore, from the equation, it is clear that the sequence $\{w_M\}$ is bounded in $W^{1,p}_0(\Om)$ with respect to $M$. Noticing the fact that the right hand side of \eqref{eq75} is bounded from above by $2\ga^{p-1} (1+|w_M|^{p^*-1})$ and the terms involving gradient of $w_M$ on the left are nonnegative, a standard Moser iteration procedure, e.g., see Ladyzhenskaya and Ura\'ltseva \cite{ladyzh}, implies that there exists a constant $C_1>0$, independent of $M$, such that 
	\begin{align}\label{eq76}
	\| w_M \|_{L^\infty(\Om)} \leq C_1.
	\end{align}
	\textit{Step (ii)}: We will prove that there exists $\Ga>0$, independent of $M$, such that $v_M \leq \Ga d^{p/(p-1+\de)}$.\\
	For this, set $\ov w(x)= M^{1/(p-1+\de)}\Ga d(x)^{p/(p-1+\de)}$ and then a simple computation gives us
	\begin{align*}
	-\De_p \ov w - \De_q \ov w \geq M (\ov w)^{-\de} \quad\mbox{in }\Om_\varrho,
	\end{align*} 
	and taking into account \eqref{eq76}, we can choose $\Ga>0$ independent of $M>0$ such that 
	\[ u_M \leq M^{1/(p-1+\de)} C_1^{1/\ga}\leq M^{1/(p-1+\de)}\Ga d^{p/(p-1+\de)} \quad\mbox{in }\Om\setminus\Om_\varrho. \]  
	Then, the weak comparison principle gives us $u_M \le M^{1/(p-1+\de)}\Ga d^{p/(p-1+\de)}$, i.e., $v_M \le \Ga d^{p/(p-1+\de)}$. \\
	\textit{Step (iii)}: Arguing as in the case (a) and using \cite[Theorem 1.8]{JDS}, we get that $\{v_M\}$ is uniformly bounded in $C^{0,\al}(\ov\Om)$ w.r.t. $M$. Therefore, up to a subsequence, $v_M\ra v$ uniformly in $K$, as $M\ra\infty$, for every compact subset $K$ of $\Om$. Moreover, we see that $v$ is the solution of 
	\begin{align*}
	-\De_p v = v^{-\de}, \ v>0   \ \ \mbox{in }\Om; \quad
	v&=0 \ \mbox{ on }\pa\Om,
	\end{align*}
	since the above equation possesses unique solution. This implies that $u_M\ra \infty$ uniformly in $K$, as $M\ra\infty$, for every compact subset $K$ of $\Om$.\\
To complete the proof, we set $\ov w = k_2 d(x)^\tau$ with $\tau={p/(p-1+\de)}$. Then, it is easy to observe that 
{\small \begin{align*}
-\De_p \ov w - \De_q \ov w &= \frac{k_2^\de}{\ov w^\de}\Big[(k_2\tau)^{p-1}\big(-\De d \; d+(p-1)(1-\tau)\big)+ (k_2\tau)^{q-1} \big(-\De d \; d+(q-1)(1-\tau)\big) \Big]  \\
&\le C \frac{k_2^\de}{\ov w^\de} := \frac{c(k_2)}{\ov w^\de} \qquad \mbox{in }\Om_\varrho,
\end{align*} }
where we have used $|\De d|\le D$ in $\Om_\varrho$. 
Since  $u_M \ra \infty$ uniformly in $K$, as $M\ra\infty$, for every compact subset $K$ of $\Om$, there exists $M_1>0$ such that $M_1 \geq c(k_3)$ and for all $M\ge M_1$
\begin{align*}
u_M \geq k_2 d(x)^\tau \quad\mbox{ in }\Om\setminus\Om_\varrho.
\end{align*}
Now, proceeding as in case (a), we complete the proof for the case $\de>1$.\\
For the case $\de=1$, the proof differs only in proving counterparts of the inequalities \eqref{eq78} and \eqref{eq79}. For these, we take $\ul w=k_1 d\log(A/d)$ and $\ov w= k_2 d\log(A/d)$, where $A$ is large enough constant satisfying $A>>\mathrm{diam}(\Om)$. Then, proceeding as above and using the bounds as in \cite[Lemma 2.3]{JDS}, we complete the prove of \eqref{eq78} and \eqref{eq79}. \QED
\end{proof}
\begin{Corollary}
	For every $M\ge 1$ and $l,L\ge 0$, there exists $\tl u_M\in W^{1,p}_0(\Om)\cap\mc C_\de$, solution to the following equation:
	\begin{equation}\label{eq77}
	\left\{\begin{array}{rllll}
	-\Delta_{p}u -\Delta_{q}u & = M u^{-\de}+l u^{q-1}+L, \; u>0 \text{ in } \Om; \quad u&=0 \  \text{ on }  \pa\Om. 
	\end{array}
	\right.
	\end{equation}
	Moreover, for sufficiently large $M$, the following holds
	\begin{align}\label{eq80}
	 \tl u_M(x) \geq \begin{cases}
	    k_2 d(x)^\frac{p}{p-1+\de}  \ \mbox{ if }\de>1\\
	    k_2 d(x)\log(A/d(x)) \ \mbox{ if }\de=1
	\end{cases}\quad\mbox{in }\Om,
	\end{align}
	where the constant $k_2$ appears in \eqref{init} and $A$ is large enough.
\end{Corollary}
\begin{proof}
  Since $1<q<p$ and $M>0$, using the standard minimization technique and employing \cite[Lemma 2.2]{JDS} together with proof of \cite[Theorem 1.4]{JDS}, we get the existence of unique solution  $\tl u_M\in W^{1,p}_0(\Om)\cap\mc C_\de$ of \eqref{eq77} (uniqueness follows from Theorem \ref{cmp}). Furthermore, it is easy to observe that $\tl u_M$ is a supersolution of \eqref{eq71}, hence by weak comparison principle (Theorem \ref{cmp}), we get $\tl u_M \ge u_M$ for all $M\ge 1$, where $u_M$ is the unique solution of \eqref{eq71}. Therefore, the required result follows from part (b) of the Theorem.\QED
 \end{proof}

 \textbf{Proof of Theorem \ref{thm1}}: We distinguish the following cases:\\
 \textit{Case (i)}: When $\de<1$.\\
  For $\la>0$, we define the energy functional associated to problem $(S_\la)$ as follows
  \begin{align*}
  	I_\la(u):= \frac{1}{2} \int_{\Om} u^2 dx+\frac{\la}{p} \int_{\Om} |\na u|^p dx+\frac{\la}{q} \int_{\Om} |\na u|^q dx-\frac{\la \vth}{1-\de} \int_{\Om} u_+^{1-\de} dx-\int_{\Om} hu_+ dx.
  \end{align*}
 We see that the first term in $I_\la$ is well defined for all $u$ in $W^{1,p}_0(\Om)$ for the case $p\ge 2N/(N+2)$ and in $W^{1,p}_0(\Om)\cap L^2(\Om)$ for $1<p<2N/(N+2)$. Hence, functional $I_\la$ is well defined in $X$, where $X= W^{1,p}_0(\Om)$ if $p\ge 2N/(N+2)$ and $X=W^{1,p}_0(\Om)\cap L^2(\Om)$, if $1<p<2N/(N+2)$. Furthermore, due to the fact $\de<1$, it is easy to prove that $I_\la$ is continuous, coercive in $X$ and strictly convex on the positive cone of $X$. Therefore, there exists a unique minimizer $u_\la\in X$ of $I_\la$. Moreover, since $I_\la(u_+)\le I_\la(u)$ for all $u\in X$, we may assume $u_\la\ge 0$ a.e. in $\Om$. Next, we will show that this $u_\la$ is in fact a weak solution of $(S_\la)$.\\
 We first construct a sub solution to problem $(S_\la)$ with the help of family of function $\tilde{u}_\rho$ obtained in Lemma \ref{lemsub}. For $\la>0$, there exists $\rho_\la>0$, sufficiently small such that 
 \begin{align}\label{eq20}
 	\tilde{u}_{\rho_\la} -\la \big(\De_p \tilde{u}_{\rho_\la}+\De_q \tilde{u}_{\rho_\la} + \vth \tilde{u}_{\rho_\la}^{-\de} \big) < -\|h\|_{L^\infty(\Om)}\leq h \quad \mbox{in } \Om.
 \end{align}
 where $\tilde{u}_{\rho_\la}\in C^1_0(\ov\Om)$ is the unique solution of \eqref{eqS}. Set $v_\la:= (\tilde{u}_{\rho_\la}-u_\la)_+$ and $\Theta(t):= I_\la(u_\la+tv_\la)$ for $t\ge 0$. 
 Then applying the arguments of \cite[P. 5049]{badra} and using \eqref{eq20}, we have $c_{\rho_\la} d\leq \tilde{u}_{\rho_\la}\le u_\la$ which on account of \cite[Lemma A2]{giacomoni} implies that $I_\la$ is G\^ateaux differentiable. Consequently, ${u}_{\la}$ is a weak solution of $(S_\la)$. Next, choosing $C>0$ large enough so that $\|h\|_{L^\infty(\Om)}< C-\la C^{-\de}$ and using weak comparison principle, we get 
 $ u_\la\le C, $
 that is, $u_\la\in L^\infty(\Om)$. To complete the proof in this case, we need to prove $u_\la\in\mc C_\de$. For this purpose, we consider $u_M\in W^{1,p}_0(\Om)\cap\mc C_\de$, the unique solution to $(\mc M)$ given by proposition \ref{sup} for $\la L > \|h\|_{L^\infty(\Om)}$ and $M>\max\{1, \la \vth\}$. Then, it is easy to verify that $u_M$ is a super solution to problem $(S_\la)$. Therefore, by comparison principle, we get $u_\la \le u_M$ in $\Om$, consequently $u_\la\in\mc C_\de$. This completes the proof of Theorem for the case $\de<1$. \\
 \textit{Case (ii)}: When $1\le\de<2+ 1/(p-1)$. \\
 In this case, with the help of Proposition \ref{prop3}, it is easy to observe that $\ov u:=\tl u_M\in W^{1,p}_0(\Om)\cap\mc C_\de$ satisfying \eqref{eq77} with $M>\max\{1, \la \vth\}$, $L=\|h\|_{L^\infty(\Om)}$ and $l=0$, is a super solution of $(S_\la)$. To construct a subsolution, we take $\rho<\vth/2$ in \eqref{eq72} and $\ul u:=v_\rho$, for $\rho>0$ sufficiently small such that $ \| v_\rho\|_{L^\infty(\Om\setminus\Om_\varrho)}^\de\leq \frac{\vth}{2\|h\|_{L^\infty(\Om)}}$ for $\varrho>0$ small enough satisfying $ \varrho^\de \leq (2k_2^\de \vth^{-1}\|h\|_{L^\infty(\Om)})^{-1}$. The above choice of $\rho$ and $\varrho$ ensures that $(\vth-\rho) v_\rho^{-\de}\leq -\|h\|_{L^\infty(\Om)}$, that is, $v_\rho$ is a subsolution of $(S_\la)$. Then, noticing $\de<2+1/(p-1)$, a truncation argument similar to proposition \ref{sup} and weak comparison principle of \cite[Theorem 1.5]{JDS} proves the existence of a solution $u\in W^{1,p}_0(\Om)\cap\mc C_{\de}$ of $(S_\la)$ for all $\vth>0$ and  $\de\in(1,2+1/(p-1))$.
The strict monotonocity of the operator in $(S_\la)$ implies the uniqueness of $u$.\QED
 \textbf{Proof of Theorem \ref{thm3}}: The standard minimization technique (see \cite[Lemma 2.1]{JDS}) shows that there exists a unique solution $u_\e\in C^{1,\sigma}(\ov\Om)$, for some $\sigma\in(0,1)$, to the following auxiliary problem: 
 \begin{equation*}
  (S_\e)\; \left\{\begin{array}{rllll}
  u_\e-\la\big(\Delta_{p}u_\e +\Delta_{q}u_\e+ \vth (u_\e+\e)^{-\de}\big) & = h, \; u_\e>0 \text{ in } \Om; \quad u_\e&=0 \ \text{ on }  \pa\Om.
 \end{array}
 \right.
 \end{equation*}
 Moreover, following \cite[Lemma 2.2]{JDS}, with
\begin{equation*}
\begin{aligned}
 \ul u_\e(x)=\eta \big(\big( d(x)+ \e^\frac{p-1+\de}{p} \big)^{\frac{p}{p-1+\de}} -\e\big) \; \; \mbox{ and }
 \ov u_\e(x)=\Gamma \big(\big( d(x)+ \e^\frac{p-1+\de}{p} \big)^{\frac{p}{p-1+\de}}-\e\big)
\end{aligned}
\end{equation*}
 for $\eta>0$ small and $\Ga$ large enough, we can prove $\ul u_\e\le u_\e\le \ov u_\e$. Next, we take $u_\e^\ga$ as a test function in the weak formulation of problem $(S_\e)$ for some $\ga>0$,
\begin{align}\label{eq59}
 \int_{\Om} u_{\e}^{1+\ga} +\int_{\Om}|\na u_\e|^{p-2} \na u_\e \na u_\e^\ga + \int_{\Om} |\na u_\e|^{q-2} \na u_\e\na u_\e^\ga =\vth \int_{\Om}  u_\e^\ga\big( u_\e +\e\big)^{-\de}+\int_{\Om} h u_\e^\ga.
\end{align}
We then observe that 
\begin{align*}
\int_{\Om} |\na u_\e|^{p-2} \na u_\e \na u_\e^\ga = \ga\Big(\frac{p}{p+\ga-1} \Big)^p \int_{\Om} |\na u_\e^{(p+\ga-1)/p}|^p
\end{align*}
and similar result holds for the third term on the left of \eqref{eq59}. Therefore, using the relation $u_\e\le \ov u_\e$, for $\ga\ge\de$, 
 from \eqref{eq59}, we infer that
\begin{align*}
\ga\Big(\frac{p}{p+\ga-1} \Big)^p \int_{\Om} |\na u_\e^{(p+\ga-1)/p}|^p \le C \int_{\Om} \ov u_\e^{(\ga-\de)} dx+\|h\|_{L^\infty(\Om)} \int_{\Om} \ov u_\e^{\ga} dx.
\end{align*}
Furthermore, noticing $\ov u_\e\le\Ga d^\frac{p}{p-1+\de}$, we get that the right side quantity is finite if and only if $\ga>\de- \frac{p-1+\de}{p}$. Thus, ${u_\e}^\rho\in W^{1,p}_0(\Om)$ is uniformly bounded for all $\rho>  \frac{(p-1)(p-1+\de)}{p^2}$. Therefore, on account of embedding results of $W^{1,p}_0(\Om)$, we can extract a subsequence, still denoting by $u_\e$, such that $u_\e(x)\ra u(x)$ a.e. in $\Om$, for some $u\in W^{1,p}_{loc}(\Om)$. By the local H\"older regularity result of Lieberman \cite[Theorem 1.7]{liebm91}, we infer that the sequence $u_\e$ converges to $u$ in $C^1_{loc}(\Om)$. Therefore, $u$ satisfies equation $(P)$ in the sense of distribution. Moreover, from the relation $\ul u_\e\le u_\e\le \ov u_\e$ and passing to limit $\e\ra 0$, we deduce that 
\begin{align}\label{eq60}
 \eta d(x)^{\frac{p}{p-1+\de}}  \le u(x) \le \Ga  d(x)^{\frac{p}{p-1+\de}}.
\end{align}
Repeating the proof of boundedness of the sequence $\{u_\e^\rho\}$ and using above comparison estimate, we see that $u^{(p+\de-1)/p}\in W^{1,p}_0(\Om)$ and  taking $\lim_{x\ra x_0\in\pa\Om}u(x)$, we get $u\in C_0(\ov\Om)$, thus $u\in\mc C_\de$. Next, we will show that $u\not\in W^{1,p}_0(\Om)$. On the contrary, we assume $u\in W^{1,p}_0(\Om)$, then from the weak formulation, it is clear that $\int_{\Om} u^{1-\de}<\infty$. Using \eqref{eq60}, we obtain a contradiction. This completes proof of the Theorem.\QED

Next, we prove the following weak comparison principle for singular problems.
\begin{Theorem}\label{cmp}
  Let $1<q<p<\infty$ and $g:\Om\times \mb R^+\to\mb R$ be a Carath\'eodory function bounded from below such that it satisfies the condition (f2) with $g$ in place of $f$. Let $u,v\in L^\infty(\Om)\cap W^{1,p}_0(\Om)$ be such that $u,v>0$ in $\Om$, $\int_{\Om} u^{1-\de}dx<\infty$, $\int_{\Om} v^{1-\de}dx<\infty$ and the following holds, in weak sense, in $\Om$:
	\begin{align*}
	-\De_p u -\De_q u \le u^{-\de}+g(x,u) \quad \mbox{and } -\De_p v -\De_q v \ge v^{-\de}+g(x,v),
	\end{align*}
 Furthermore, suppose that there exists a  positive function $w\in L^\infty(\Om)$ and constants $c_1, c_2>0$ such that $c_1 w\le u, v \le c_2 w$ with 
	\begin{align}\label{eq51}
	\int_{\Om} |g(x,c_1 w)| w dx <\infty \quad \mbox{and }\int_{\Om} |g(x,c_2 w)| w dx <\infty.
	\end{align}
  Then, the comparison principle holds, that is, $u\le v$ in $\Om$.
\end{Theorem}
\begin{proof}
 For $\e>0$, set $u_\e =u+\e$ and $v_\e= v+\e$. Let 
	\begin{align*}
	\phi:= \frac{u_\e^q-v_\e^q}{u_\e^{q-1}} \quad \mbox{ and } \psi:= \frac{v_\e^q-u_\e^q}{v_\e^{q-1}}. 
	\end{align*}
 We recall the following Diaz-Sa\'a inequality \cite[Theorem 2.5]{takgiac} (see also \cite[Remark 2.10]{brasco})
	\begin{align}\label{diaz}
	\int_{\Om} \left( \frac{-\De_p w_1}{w_1^{r-1}}+ \frac{\De_p w_2}{w_2^{r-1}}\right) (w_1^r-w_2^r) dx \ge 0,
	\end{align}
 for all $1<r\le p$, and the equality holds if and only if $w_1=kw_2$, for some constant $k>0$. Set $\Om_+= \{ x\in\Om : u(x)>v(x) \}$, then we have $\phi\ge 0$ and $\psi \le 0$ in $\Om_+$. Now testing the first equation by $\phi$ and the second by $\psi$, we obtain
	\begin{align*}
	\int_{\Om_+} &(-\De_p u-\De_q u) \phi + \int_{\Om_+} (-\De_p v-\De_q v)  \psi \\
	&\le \int_{\Om_+} \big(u^{-\de}+g(x,u) \big) \phi  + \int_{\Om_+} \big(v^{-\de}+g(x,v) \big) \psi.
	\end{align*}
 Rearranging the terms and using Diaz-Sa\'a inequality \eqref{diaz}, we get
	\begin{equation}\label{eq55}
	\begin{aligned}
	0\le & \int_{\Om_+} \Big(\frac{-\De_p u_\e}{u_\e^{q-1}} + \frac{\De_p v_\e}{v_\e^{q-1}} \Big) (u_\e^q-v_\e^q) + \int_{\Om_+} \Big(\frac{-\De_q u_\e}{u_\e^{q-1}} + \frac{\De_q v_\e}{v_\e^{q-1}} \Big) (u_\e^q-v_\e^q) \\
	&\le \int_{\Om_+} \big(u^{-\de}+g(x,u) \big) \phi  + \int_{\Om_+} \big(v^{-\de}+g(x,v) \big) \psi.
	\end{aligned}
	\end{equation}
 The right side quantity of above equation simplifies to 
	\begin{align*}
	\int_{\Om_+} (u^{-\de}+g(x,u) ) \phi  + \int_{\Om_+} (v^{-\de}+g(x,v)) \psi \le \int_{\Om_+} \left[ \frac{g(x,u)}{u^{q-1}} \frac{u^{q-1}}{u_\e^{q-1}} - \frac{g(x,v)}{v^{q-1}} \frac{v^{q-1}}{v_\e^{q-1}}\right] (u_\e^q -v_\e^q)dx.
	\end{align*}
 Since $\frac{u}{u_\e} \ra 1$ and $\frac{v}{v_\e} \ra 1$ as $\e\ra 0^+$ a.e. in $\Om$, using \eqref{eq51} and dominated convergence theorem, we obtain
	\begin{align}\label{eq56}
	\lim_{\e\ra 0^+} \int_{\Om_+} \big( g(x,u)\phi +g(x,v)\psi \big) \le 0.
	\end{align}
  Taking into account \eqref{diaz}, \eqref{eq55}, \eqref{eq56} and using Fatou lemma, we deduce that 
	\begin{align*}
	\int_{\Om_+} \Big(\frac{-\De_q u}{u^{q-1}} + \frac{\De_q v}{v^{q-1}} \Big) (u^q-v^q)=0,
	\end{align*}
  therefore,
  there exists $k\in\mb R$ such that $u=k v$ in $\Om_+$. Now, we will prove that $k\le 1$. On the contrary assume $k>1$, then the following hold
    {\small\begin{equation*}
		\begin{aligned}
		&k^q \int_{\Om_+} (\na v|^p+|\na v|^q) \le \int_{\Om_+} (|\na u|^p+|\na u|^q) \le \int_{\Om_+} (u^{1-\de}+g(x,u)u )= \int_{\Om_+} k^{1-\de} v^{1-\de} +g(x,kv)kv \\
		&k^q \int_{\Om_+}( |\na v|^p+|\na v|^q )\ge k^q \int_{\Om_+} (v^{1-\de}+g(x,v)v )\ge \int_{\Om_+} (k^q v^{1-\de}+g(x,kv)kv)
		\end{aligned}
		\end{equation*}}
  where in the last inequality we have used the fact that $k>1$, $v>0$ and $\frac{g(x,s)}{s^{q-1}}$ is decreasing in $s$. Therefore, we get $k^{q+\de-1}\le 1$, which yields a contradiction, since $\de>0, q>1$. This implies that $u\le v$ in $\Om_+$ and from the definition of $\Om_+$, we get $u\le v$ in $\Om$. \QED
\end{proof}

\textbf{Proof of Theorem \ref{thm4}}: We first consider the case $\de<1$. For $0<\al_f <m$, let $L>0$ be such that $-L \le f(x,s) \le m s^{q-1}+L$. Fix $\e>0$ small enough so that the unique positive solution $\tl u_\e$ of \eqref{eqS} satisfies the following inequality in weak sense,
 \begin{align*}
    -\De_p \tl u_\e -\De_q \tl u_\e -\vth \tl u_\e^{-\de} \le -L \quad \mbox{in }\Om.
  \end{align*}
 Furthermore, we take $l>m$ and $M>\max\{1,\vth\}$ in proposition \ref{sup} and denote $v\in W^{1,p}_0(\Om)\cap\mc C_\de$ as the unique positive solution of problem $(\mc M)$ in \eqref{eq52}. Then, it follows that 
  \begin{align*}
    -\De_p v -\De_q v - \vth \; v^{-\de} \ge m v^{q-1}+ L \quad \mbox{in }\Om.
  \end{align*}
We set the following 
  \begin{align*}
    h(x,u(x))=\begin{cases}
       u(x)^{-\de} \ \mbox{if }u(x)\ge \tl u_\e(x) \\
       \tl u_\e^{-\de} \quad \mbox{otherwise}
      \end{cases}
    , \; \; g(x,u(x))= \begin{cases}
      f(x,\tl u_\e(x)) \ \mbox{if }\tl u_\e(x)\le u(x) \\
      f(x, u(x)) \ \mbox{if }\tl u_\e(x)\le u(x) \le v(x) \\
      f(x,v(x)) \ \mbox{if } u(x)\le v(x).  
     \end{cases}
  \end{align*}
 Let $H(x,t):= \int_{0}^{t} h(x,s) ds$ and $G(x,t):= \int_{0}^{t} g(x,s) ds$. We define the energy functional $\mc J :W^{1,p}_0(\Om) \ra \mb R$ as follows
  \begin{align*}
    \mc J(u)= \frac{1}{p} \int_{\Om} |\na u|^p dx+\frac{1}{q} \int_{\Om} |\na u|^q dx- \vth\int_{\Om} H(x,u) dx-\int_{\Om} G(x,u) dx.
 \end{align*}
 It is easy to verify that $\mc J$ is coercive and weakly lower semicontinuous in $W^{1,p}_0(\Om)$. Similar to the proof of proposition \ref{sup}, we obtain a global minimizer $u$ of $\mc J$ in $W^{1,p}_0(\Om)$. Again by using \cite[Lemma A2]{giacomoni}, we get that $\mc J$ is G\^ateaux differentiable and thus $u$ satisfies 
 \begin{align*}
    \left\{-\De_p u -\De_q u = \vth \; h(x,u) +g(x,u) \ \ \mbox{ in }\Om; \quad
      u=0 \ \mbox{ on }\pa\Om. \right.
 \end{align*}
 Then, by repeated application of weak comparison principle, first we can conclude that $\tl u_\e \le u$, thus $h(x,u)= u^{-\de}$ and then $u\le v$. Therefore, $u$ satisfies $(P)$ in the weak sense. From the above procedure it is clear that $u\in\mc C_\de$. \par
 Next, we consider the case $\de\ge 1$. Since $f$ is uniformly local Lipschitz function, there exists $K>0$ such that the map $t\mapsto f(x,t)+ Kt$ is nondecreasing in $[0,\|\ov u\|_{L^\infty(\Om)}]$, where $\ov u$ is specified below. Consider the following iterative scheme
 \begin{align}\label{eq63}
    -\De_p u_n -\De_q u_n - \vth \; u_n^{-\de} + Ku_n &= f(x,u_{n-1})+ K u_{n-1} \quad\mbox{in }\Om,\nonumber\\
     u_n &=0 \quad\mbox{on }\pa\Om,
  \end{align}
 with $u_0:=\ul u$. Now, we specify the choices for $\ul u$ and $\ov u$. Since $f(x,s)\ge -L$, we repeat the process of constructing subsolution of Theorem \ref{thm1} for the case $\de\ge 1$ with $L=\|h\|_{L^\infty(\Om)}$. Therefore, for sufficiently small $\rho>0$, we set $\ul u=v_\rho$. For the supersolution, we take $M>\max\{1,\vth\}$ and $l=m$ in \eqref{eq77} and set $\ov u= \tl u_M$. By the choice of $K$ and Theorem \ref{thm1}, we see that the scheme is well defined and produces a sequence $\{u_n\}\subset W^{1,p}_0(\Om)\cap\mc C_\de\cap C_0(\ov\Om)$ as solution to problem \eqref{eq63}. Using weak comparison principle (Theorem \ref{cmp}), we obtain $\ul u \le u_n \le \ov u$ and again using this together with the monotonocity of the map $t\mapsto f(x,t)+Kt$, we get that the sequence $\{u_n\}$ is monotone increasing. Moreover, as a consequence of \cite[Theorem 1.8]{JDS} $u_n\in C^{0,\sigma}(\ov\Om)$ and from the relation $\ul u \le u_n \le \ov u$, we get that the sequence $\{u_n\}$ is uniformly bounded in $\mc C_\de\cap C_0(\ov\Om)$ with respect to $n$. Therefore, by Arzela-Ascoli theorem, we get that $u_n\ra u$ in $\mc C_\de\cap C_0(\ov\Om)$, as $n\ra\infty$. This together with the weak formulation of \eqref{eq63} implies that the sequence $\{u_n\}$ is Cauchy in $W^{1,p}_0(\Om)$ and therefore converges to $u$ in $W^{1,p}_0(\Om)$. Then, passing to the limit as $n\ra\infty$ in the equation \eqref{eq63} and using Lebesgue dominated convergence theorem, we get that $u$ is a solution of problem $(S_\la)$.
Furthermore, the uniqueness of the solution follows by Theorem \ref{cmp}.   \QED
Now, we prove some higher Sobolev integrability result for the following equation:
 \begin{align}\label{eq81}
 	\mathrm{div}A(x,\na u) = \mathrm{div}f \quad\mbox{in }\Om,
 \end{align}
where $f\in L^{p/(p-1)}(\Om;\mb R^N)$, $\Om$ is a bounded domain and $A: \ov\Om\times\mb R^N \to \mb R^N$ satisfying the following:
 \begin{enumerate}
	\item[(A1)] $|A(x,z)|+|\pa_z A(x,z) z| \le \La (1+|z|^{p-1})$,
	\item[(A2)] $z.A(x,z)  \ge \nu |z|^p $,
	\item[(A3)] $\ds\sum_{i=1}^{N} |A^i(x,z)- A^i(y,z)| \le \La (1+|z|^{p-1}) |x-y|^\omega$,
 \end{enumerate}
where $0<\nu\le\La$ are constants. We have the following notation
 \[ D^m(\Om):=\{ g\in L^m(\Om;\mb R^N) : \ \exists \; u\in L^1_\mathrm{loc}(\Om) \mbox{ with } g=\na u \}. \]
 We state our theorem in this regard as follows.
\begin{Theorem}
 Let $f\in L^{p/(p-1)}(\Om;\mb R^N)$ and $u$ be a weak solution of \eqref{eq81} such that $\na u\in D^p(\Om)$. Suppose $f\in L^{m/(p-1)}(\Om;\mb R^N)$ for some $m\ge p$, then $\na u$ belongs to $L^m(\Om;\mb R^N)$ and 
  \begin{align}\label{eq82}
  	\| \na u \|_{L^m}^{p-1} \leq C_m \big(\| f \|_{L^{m/(p-1)}}+1\big),
  \end{align}
  where $C_m>0$ is a constant which depends only on $N,m,p,\Om$ and $\La,\nu,\omega$.
 \end{Theorem}
\begin{proof}
 Proof of the theorem follows by following the steps of \cite[Theorem 1.1]{dibMan}(see also \cite{caff}). Indeed, let $w\in W^{1,p}(B_R)$ be the unique solution to the problem:
  \begin{align*}
  	\mathrm{div} A(0,\na w)=0 \ \ \mbox{ in }B_R; \quad w=u \ \ \mbox{ on }\pa B_R,
  \end{align*}
  then, following the procedure of \cite{dibMan}, one can prove that: \\
  For $\eta\in(0,1)$, there exists a constant $C_{\eta}>0$ such that 
  \begin{align*} 
  \Xint-_{B_\rho}|\na(u-w)|^p \le \eta \Big(\frac{R}{\rho}\Big)^N \Xint-_{B_R}|\na w|^p + C_\eta \Big(\frac{R}{\rho}\Big)^N \Xint-_{B_R}|f|^{p'} \quad\mbox{for }0<\rho\le R.
  \end{align*} 
 Using the above inequality and estimates of \cite[Lemma 4.1]{JDS}, we have\\
 For $\eta\in(0,1)$, there exist $C_{\eta}>0$ and $h,\varsigma\in(0,1)$ such that, for $0<\rho\le hR$,
 \begin{align*} 
 \Xint-_{B_\rho}|\na u-(|\na u|^p)_{\rho}|^p \le C_\eta \Big(\frac{R}{\rho}\Big)^N \Xint-_{B_R}|f|^{p'}+ \left[\eta \Big(\frac{R}{\rho}\Big)^N+ C \Big(\frac{\rho}{R}\Big)^{\varsigma p}\right] \Xint-_{B_R}|\na w|^p+ C\Big(\frac{\rho}{R}\Big)^{\varsigma p},
 \end{align*}
 where $(v)_{\rho}=\Xint-_{B_\rho}v~dx$. Now, proof of \eqref{eq82} is a standard procedure by making the use of maximal operators and can be completed as in \cite[Proof of Theorem 1.1, pp 1113-1116]{dibMan}.\QED

\end{proof}
\begin{Corollary}\label{cor2}
 Let $1<q<p$, $p\ge 2$ and $u$ be a solution to the problem: 
  \begin{align*}
  	\left\{  \ -\De_p u -\De_q u = g(x)+b(x) \ \mbox{ in }\Om; \quad u=0 \mbox{ on }\pa\Om,      \right.
  \end{align*}
  where $b\in L^\infty(\Om)$, $c_1 d(x)^{-\al} \le g(x) \le c_2 d(x)^{-\al}$ with $\al\in(1,2)$ and $c_1,c_2>0$ are constants. Then, $|\na u|\in L^m(\Om)$ for $ m < (p-1)/(\al-1)$.
\end{Corollary}
\begin{proof}
 We first consider the following problem:
  \begin{align*}
  \left\{  \ -\De w = g(x)+b(x) \ \mbox{ in }\Om; \quad w=0 \ \mbox{ on }\pa\Om.      \right.
  \end{align*}
 Since the domain is $C^2$ regular, there exists Green function $G(x,y)$ for $-\Delta$ in $\Om$. Furthermore, there exists a positive constant $c$ such that the following holds, for $x,y\in\Om$ (see \cite{zhao}),
 \begin{align*}
 	\frac{1}{c |x-y|^N} \min\{|x-y|^2, d(x)d(y)\}\le G(x,y) \le \frac{c}{|x-y|^N}\min\{|x-y|^2, d(x)d(y)\}.
 \end{align*}
 The solution $w$ is represented by $w(x):=\int_{\Om} G(x,y)g(y)b(y)dy$. Then, using the behaviour of $g$, it is not difficult to prove that 
 \begin{equation*}
 \left\{\begin{array}{rllll}
 &c_3 \; d(x)\le w(x)\le c_4 \; d(x) \qquad \mbox{if }\al<1,\\
 &c_7 \; d(x)^{2-\al} \le w(x) \le c_8 \; d(x)^{2-\al} \quad \mbox{if }\al\in(1,2).
 \end{array}
 \right.
 \end{equation*}
Proceeding similar to the proof of \cite[Theorem 2.5]{crandall}, we can prove that $|\na w|\le c_9 d(x)^{1-\al}$, for some positive constant $c_9$. Therefore, taking $f= -\na w$, we see that $f\in L^{m/(p-1)}(\Om;\mb R^N)$ for all $m<(p-1)/(\al-1)$. Hence, the required result follows from \eqref{eq82} for the choice of $A(\xi):=|\xi|^{p-2}\xi+|\xi|^{q-2}\xi$ with $\xi\in\mb R^N$. \QED
\end{proof}
\textbf{Proof of Theorem \ref{sob-int}}: Let $u\in W^{1,p}_{loc}(\Om)$ be the solution of problem $(PS)$, whose existence and uniqueness is guaranteed by similar steps of Theorems 1.4 and 1.5, respectively of \cite{JDS}. Moreover, $u$ satisfies the following
 \begin{align*}
 	c_1 d(x)^{p/(p-1+\de)}\le u(x) \le c_2 d(x)^{p/(p-1+\de)},
 \end{align*}
for some positive constants $c_1,c_2$. Indeed, one can construct a sub and super solution $\ul u$ and $\ov u$ of $(PS)$, respectively using the estimates of \cite[Lemma 2.2]{JDS} as below
 \begin{align*}
 	-\De_p \ul u -\De_q \ul u  \leq \ul u^{-\de} - \| b \|_{L^\infty(\Om)} \\
 	-\De_p \ov u -\De_q \ov u  \geq \ov u^{-\de} + \| b \|_{L^\infty(\Om)}.
 \end{align*} 
 Then, we observe that
 $u^{-\de}$ behaves as $d(x)^{\frac{-p\de}{p-1+\de}}$ with $\frac{p\de}{p-1+\de}<2$. Let $v$ be the solution of the following problem, as obtained in corollary \ref{cor2},
 \begin{align*}
 \left\{  \ -\De_p v -\De_q v = u^{-\de}+b(x) \ \mbox{ in }\Om; \quad u=0 \mbox{ on }\pa\Om.      \right.
 \end{align*}
By the uniqueness of solution to $(PS)$, we see that $v=u$ in $\Om$. Therefore, employing the result of corollary \ref{cor2}, we get that $|\na u|\in L^m(\Om)$ for all $m<\frac{p-1+\de}{\de-1}$. Noting this and the fact that $u\in C_0(\ov\Om)$, we get the required result of the theorem.\QED

\section{Existence and regularity of solution to $(G_t)$}
 In this section, we obtain the existence result for problem $(G_t)$ with the help of Theorem \ref{thm1} and using semi-discretization in time with implicit Euler method.\\
 \textbf{Proof of Theorem \ref{thm2}}: Fix $N_0\in\mb N$ and set $\De_t=\frac{T}{N_0}$. For $0\le n\le N_0$, define $t_n=n\De_t$ and $g^n(x)=\frac{1}{\De_t}\int_{t_{n-1}}^{t_n} g(x,\tau) d\tau$. Since $g\in L^\infty(Q_T)$, we have $g^n\in L^\infty(\Om)$. Set 
 \[ g_{\De_t}(x,t)= g^n(x), \quad \mbox{if }t\in[t_{n-1},t_n), \mbox{ for } 1\le n\le N_0.  \]
 It follows from Jensen inequality that, for $1<r<\infty$,
  \begin{align}\label{eq41}
  	\|g_{\De_t}\|_{L^r(Q_T)}\le \big( T|\Om| \big)^\frac{1}{r} \|g\|_{L^\infty(Q_T)}
  \end{align}
 and $g_{\De_t}\ra g$ in $L^r(Q_T)$ as $\De_t\ra 0$. We take $\la=\De_t, h= \De_t g^n+ u^{n-1}\in L^\infty(\Om)$ in $(S_\la)$ and define iteratively, $u^n\in W^{1,p}_0(\Om)\cap\mc C_\de$ with the following implicit Euler scheme,
  \begin{align}\label{eq38}
  	\frac{u^n-u^{n-1}}{\De_t} -&\De_p u^n-\De_q u^n -\vth (u^n)^{-\de}= g^n \quad \mbox{in }\Om, \nonumber\\
  	&u^n=0 \qquad \mbox{on }\pa\Om,
  \end{align}
 where $u^0=u_0\in W^{1,p}_0(\Om)\cap\mc C_\de$. Then, for all $n\in\{1,\dots, N_0 \}$ and $t\in[t_{n-1}, t_n)$, we set
 \begin{equation}\label{eq30}
  \begin{aligned}
  	u_{\De_t}(\cdot,t):= u^n(\cdot) \quad\mbox{and } \ 
  	\tl u_{\De_t}(\cdot,t):=\frac{(t-t_{n-1})}{\De_t} \big(u^n(\cdot)-u^{n-1}(\cdot)\big)+u^{n-1}(\cdot).
  \end{aligned}
  \end{equation}
 From the above definition, it is easy to observe that 
  \begin{align}\label{eq31}
  	\frac{\pa \tl u_{\De_t}}{\pa t}-\De_p u_{\De_t} -\De_q u_{\De_t} - \vth u_{\De_t}^{-\de} =g_{\De_t}\in L^\infty(Q_T).
  \end{align}
 Now, multiplying \eqref{eq38} by $\De_t u^n$ and summing from $n=1$ to $N_1\le N_0$, for $\e>0$, using Young inequality, we get 
 \begin{align*}
 	\sum_{n=1}^{N_1} \int_{\Om} (u^n-u^{n-1}) u^n +&\De_t \left( \sum_{n=1}^{N_1} \|u^n\|^p_{W^{1,p}_0(\Om)}+\sum_{n=1}^{N_1} \|u^n\|^q_{W^{1,q}_0(\Om)} -\vth \sum_{n=1}^{N_1} \int_{\Om} (u^n)^{1-\de}\right) \\ &\le C(\e) \|g\|_{L^\infty(\Om)}^{p'} +\e \De_t \sum_{n=1}^{N_1} \|u^n\|^p_{W^{1,p}_0(\Om)},
 \end{align*}
 where $C(\e)>0$ is a constant and $p'=p/(p-1)$. An easy manipulation yields
  \begin{align*}
    \sum_{n=1}^{N_1} \int_{\Om} (u^n-u^{n-1}) u^n= \frac{1}{2} \left( \sum_{n=1}^{N_1} \int_{\Om} |u^n-u^{n-1}|^2 +\int_{\Om} (u^{N_1})^2-\int_{\Om} u_0^2 \right).
  \end{align*} 
  To control the integrals involving singular term, we will construct suitable sub and super solution, 
  $\ul u$ and $\ov u$ in $W^{1,p}_0(\Om)\cap\mc C_\de$ such that $\ul u\le u_0 \le \ov u$ and the following hold in the weak sense in $\Om$,
 \begin{align}\label{eq40}
  \begin{cases}
 	-\De_p \ul u- \De_q \ul u -\vth \ul u^{-\de} \le -\|g\|_{L^\infty(Q_T)} \\
 	-\De_p \ov u- \De_q \ov u - \vth \ov u^{-\de} \ge \|g\|_{L^\infty(Q_T)}.
 	\end{cases}
 \end{align} 
Indeed, since $u_0\in\mc C_\de$, there exist positive constants $k_1,k_2$ such that $k_1 d(x)\le u_0(x)\le k_2 d(x)$ in $\Om$, for $\de<1$. On account of  corollary \ref{cor1}, we can choose $\rho>0$ sufficiently small in Lemma \ref{lemsub} such that $\tl u_\rho \le u_0$ and the first equation of \eqref{eq40} holds with $\ul u=\tl u_\rho$, where $\tl u_\rho$ is the unique positive solution of \eqref{eqS}. To obtain a suitable supersolution, we invoke proposition \ref{sup}. Let $u_M\in W^{1,p}_0(\Om)\cap\mc C_\de$ be the unique solution of problem $(\mc M)$ in \eqref{eq52} with $l=0$ and $M>\max\{1,\vth\}$. Since $\frac{u_M}{d}\ra\infty$ uniformly in $\Om$ as $M\ra\infty$, we can choose $\ov u=u_M$, for $M$ sufficiently large, such that the second equation in \eqref{eq40} holds and $u_0\leq \ov u=u_M$ in $\Om$. Whereas for the case $\de\ge 1$, we construct subsolution exactly as in the proof of Theorem \ref{thm1} with $L=\|g\|_{L^\infty(Q_T)}$ with an additional assumption that $\rho>0$ is taken sufficiently small such that \eqref{eq78} holds. That is, $\ul u=v_\rho$, where $v_\rho$ is the solution of \eqref{eq71} and $\rho>0$ sufficiently small such that $\ul u\le u_0$ (a consequence of \eqref{eq78}). While for the supersolution, we take $l=0$, $L=\|g\|_{L^\infty(Q_T)}$ and $M$ large enough in \eqref{eq77} such that the solution $\tl u_M:=\ov u$ satisfies \eqref{eq80}, that is $\ov u\ge u_0$.
 Then, repeated application of weak comparison principle gives us $\ul u\le u^n \le \ov u$ and thus 
 \begin{align}\label{eq46}
 \ul u \le u_{\De_t}, \tl u_{\De_t} \le \ov u. 
 \end{align}
 Therefore, 
 \begin{align*}
    \De_t \sum_{n=1}^{N_1} \int_{\Om} (u^n)^{1-\de}\le \begin{cases}
        T\int_{\Om} \ov u^{1-\de} <\infty \quad \mbox{if } \de\le 1,\\
         T\int_{\Om} \ul u^{1-\de} <\infty \quad \mbox{if } \de>1.
    \end{cases}
 \end{align*}
 Resuming all the information, we conclude that the sequences $u_{\De_t}, \tl u_{\De_t}\in\mc C_\de$ uniformly and are bounded in $L^p(0,T; W^{1,p}_0(\Om))\cap L^\infty(0,T; L^\infty(\Om))$. Proceeding similar to \cite[(2.10)-(2.12), Page 5054]{badra}, we obtain
 \begin{equation}\label{eq32}
  \begin{aligned}
    \frac{\De_t}{2} \sum_{n=1}^{N_1} \int_{\Om} \Big(\frac{u^n-u^{n-1}}{\De_t}\Big)^2 &+\frac{1}{p}  \int_{\Om} |\na u^{N_1}|^p -\frac{1}{p}\int_{\Om} |\na u_0|^p  +\frac{1}{q}  \int_{\Om} |\na u^{N_1}|^q -\frac{1}{q}\int_{\Om} |\na u_0|^q \\
     & -\frac{\vth}{1-\de} \left( \int_{\Om} |u^{N_1}|^{1-\de} -\int_{\Om} |u_0|^{1-\de} \right) \le \frac{|\Om| T}{2} \|g\|^2_{L^\infty(Q_T)}.
  \end{aligned}
 \end{equation}
 From above expression together with the fact $\int_{\Om} (u^n)^{1-\de}\le \max\{ \int_{\Om} \ov u^{1-\de}, \int_{\Om} \ul u^{1-\de}  \}$, we infer that 
  \begin{align}\label{eq33}
  	&\frac{\pa \tl u_{\De_t}}{\pa t} \text{ is bounded in } L^2(Q_T) \text{ uniformly in }\De_t, \\ \label{eq34}
  	u_{\De_t}, \tl u_{\De_t} &\text{ are bounded in } L^\infty(0,T; W^{1,p}_0(\Om)) \text{ uniformly in }\De_t.
  \end{align}
 On account of the preceding observation and \eqref{eq32}, an easy computation yields
  \begin{align}\label{eq35}
  	\| u_{\De_t}- \tl u_{\De_t} \|_{L^\infty(0,T;L^2(\Om))} \le \max_{1\le n\le N_1} \|u^n -u^{n-1} \|_{L^2(\Om)} \le C (\De_t)^{1/2}.
  \end{align}  
 Taking $N_1\ra\infty$, that is $\De_t\ra 0^+$, from \eqref{eq33} and \eqref{eq34}, up to a subsequence, we obtain 
 \begin{equation}\label{eq42}
  \left\{\quad 
  \begin{aligned}
  	&\tl u_{\De_t} \overset{\ast}{\rightharpoonup} u \quad \mbox{ in }L^\infty(0,T; W^{1,p}_0(\Om)\cap L^\infty(\Om)), \\
  	&u_{\De_t} \overset{\ast}{\rightharpoonup} v \quad \mbox{ in }L^\infty(0,T; W^{1,p}_0(\Om)\cap L^\infty(\Om)), \\
  	&\frac{\pa \tl u_{\De_t}}{\pa t} \rightharpoonup \frac{\pa u_{\De_t}}{\pa t} \qquad\mbox{ in } L^2(Q_T),
  \end{aligned}
  \right.
  \end{equation}
 as $\De_t\ra 0^+$, for some $u,v\in L^\infty(0,T; W^{1,p}_0(\Om)\cap L^\infty(\Om))$ with $u,v\in\mc C_\de$ uniformly and $\frac{\pa u_{\De_t}}{\pa t}\in L^2(Q_T)$. With the help of \eqref{eq35} and \eqref{eq46}, we conclude that $u\equiv v$ and $\ul u\le u \le\ov u$. Therefore, $u\in \mc V(Q_T)$.
  Now we will prove that $u$ is a solution of $(G_t)$. For this, set 
 	\begin{align*}
 	  S=\big\{ u\in L^\infty(0,T; W^{1,p}_0(\Om)) : \frac{\pa u}{\pa t}\in L^2(Q_T) \big\}.
 	\end{align*}
 By Aubin-Lions-Simon lemma, we get that $S$ is compactly embedded into $C([0,T];L^2(\Om))$. Therefore, by interpolation identity, for all $r>1$, up to a subsequence, we have
  \begin{align}\label{eq45}
  	\tl u_{\De_t} \ra u \quad \mbox{in } C(0,T; L^r(\Om)) \mbox{ as } \De_t\ra 0.
  \end{align}
 Using \eqref{eq35}, for all $r>1$, we deduce that 
    \begin{align*}
     u_{\De_t} \ra u \quad \mbox{in } L^\infty(0,T; L^r(\Om)) \mbox{ as } \De_t\ra 0.	
    \end{align*}
 Multiplying \eqref{eq31} by $(u_{\De_t}-u)$ and using the above convergence results, we obtain
 \begin{align*}
 	&\int_{Q_T} \left(\frac{\pa\tl u_{\De_t}}{\pa t} -\frac{\pa u_{\De_t}}{\pa t} \right) (\tl u_{\De_t}-u) dxdt-\int_{0}^{T} \ld \De_p u_{\De_t}, u_{\De_t}-u \rd dt 
 	- \int_{0}^{T} \ld \De_q u_{\De_t}, u_{\De_t}-u \rd dt \\
 	&\quad-\vth\int_{Q_T} u_{\De_t}^{-\de} (u_{\De_t}-u) dxdt =\int_{Q_T} g_{\De_t} (u_{\De_t}-u) dxdt + o_{\De_t}(1).
 \end{align*}
 Since $\ul u\le u_{\De_t}, \tl u_{\De_t} \le \ov u$, by dominated convergence theorem, we have 
 \begin{align*}
 	\int_{Q_T} u_{\De_t}^{-\de} (u_{\De_t}-u) dxdt = o_{\De_t}(1), \quad
 	\int_{Q_T} g_{\De_t} (u_{\De_t}-u) dxdt = o_{\De_t}(1).
 \end{align*}
 Furthermore, performing integration by parts and taking into account the convergence results of $u_{\De_t}$ and $\tl u_{\De_t}$, we deduce that 
  \begin{align}\label{eq37}
  	\frac{1}{2}\int_{\Om} |\tl u_{\De_t} -u|^2(T)dx - \int_{0}^{T} \big(\ld \De_p u_{\De_t}-\De_p u, u_{\De_t}-u \rd + 
  	 \ld \De_q u_{\De_t}-\De_q u, u_{\De_t}-u \rd \big)dt = o_{\De_t}(1).
  \end{align} 
 Now, we distinguish the following cases:\\
 \textit{Case}(i): $q\ge 2$.\\
 In this case, we will use the following inequality 
 \[ |a-b|^l \le 2^{l-2} (|a|^{l-2}a-|b|^{l-2}b)(a-b) \quad \mbox{for }a,b\in\mb R^N, \mbox{ and }l\ge 2.   \]
 Then, from \eqref{eq37}, we deduce that 
 \begin{align*}
 	\int_{0}^{T} \int_{\Om} |\na( u_{\De_t}- u)|^p dxdt =o_{\De_t}(1) \quad \mbox{and }  
 	\int_{0}^{T} \int_{\Om} |\na( u_{\De_t}- u)|^q dxdt =o_{\De_t}(1).
 \end{align*}
 \textit{Case}(ii): $1<q<p<2$.\\
 In this case, we consider the following inequality
 \[ |a-b|^l \le C_l \big( (|a|^{l-2}a-|b|^{l-2}b)(a-b) \big)^\frac{l}{2} (|a|^l +|b|^l)^\frac{2-l}{2} \quad \mbox{for }a,b\in\mb R^N, \mbox{ and } 1<l< 2,   \]
 where $C_l$ is a positive constant depending only on $l$. Therefore, noting the fact that $u\not\equiv 0$ and using H\"older inequality, we obtain
 \begin{align*}
 	0\le \int_{\Om} |\na (u_{\De_t}-u)|^p &\le C_p \int_{\Om} \big( (|\na u_{\De_t}|^{p-2}\na u_{\De_t}-|\na u|^{p-2}\na u) (\na u_{\De_t}-\na u)\big)^\frac{p}{2} \\ 
 	&\qquad\qquad \times \big(|\na u_{\De_t}|^p +|\na u|^p\big)^\frac{2-p}{2} ~dx \\
 	&\le C_p \left(\int_{\Om} (|\na u_{\De_t}|^{p-2}\na u_{\De_t}-|\na u|^{p-2}\na u) (\na u_{\De_t}-\na u)\right)^\frac{p}{2} \\
 	&\qquad
 	 \times\Big( \|u_{\De_t}\|^p_{W^{1,p}_0(\Om)} + \|u\|^p_{W^{1,p}_0(\Om)}\Big)^\frac{2-p}{2}.
 \end{align*}
 This implies 
 \begin{align*}
 	\frac{ \left( \int_{\Om} |\na (u_{\De_t}-u)|^p\right) ^\frac{2}{p} }{\Big( \|u_{\De_t}\|^p_{W^{1,p}_0(\Om)} + \|u\|^p_{W^{1,p}_0(\Om)}\Big)^\frac{2-p}{p} } \le C_p \int_{\Om} (|\na u_{\De_t}|^{p-2}\na u_{\De_t}-|\na u|^{p-2}\na u) (\na u_{\De_t}-\na u)dx
 \end{align*}
  and similar result holds for the integrals with exponent $q$. 
  Thus taking into account \eqref{eq35}, we deduce that 
  \begin{align*}
    \int_{0}^{T} \int_{\Om} |\na( u_{\De_t}- u)|^p dxdt =o_{\De_t}(1) \quad \mbox{and }  
    \int_{0}^{T} \int_{\Om} |\na( u_{\De_t}- u)|^q dxdt =o_{\De_t}(1).
  \end{align*}
 \textit{Case}(iii) : $1<q\le 2 \le p$.\\
 In this case the convergence results follows from cases (i) and (ii).\\
 Furthermore, from the above discussion, it is clear that 
 \begin{align}\label{eq43}
 	-\De_p u_{\De_t} \ra -\De_p u \;\mbox{ in } L^{p'} (0,T; W^{-1,p'}(\Om)) \quad \mbox{and }  -\De_q u_{\De_t} \ra -\De_q u \;\mbox{ in } L^{q'} (0,T; W^{-1,p'}(\Om)),
 \end{align}
 where $p'$ and $q'$ are H\"older conjugate of $p$ and $q$, respectively. Moreover, from \cite[(2.23), page 5056]{badra}, we have
 \begin{align}\label{eq44}
 	u_{\De_t}^{-\de} \ra u^{-\de} \quad \mbox{in }L^\infty(0,T; W^{-1,p'}(\Om)).
 \end{align}
 Therefore, from \eqref{eq41}, \eqref{eq42}, \eqref{eq43}, \eqref{eq44} and using dominated convergence theorem, we get that $u$ is a weak solution of $(G_t)$. 
 Next, we will show that such a weak solution $u\in\mc C_\de$ is unique. On the contrary, assume $v\in\mc C_\de$ is another weak solution of $(G_t)$. Then, from the weak formulation of $(G_t)$, we have 
 \begin{align*}
 	\int_{\Om} \frac{\pa(u-v)}{\pa t}(u-v) - \ld \De_p u-\De_p v, u-v \rd - \ld \De_q u -\De_q v, u-v \rd =\vth\int_{\Om} (u^{-\de}-v^{-\de})(u-v).
 \end{align*}
 Taking into account the inequalities used in Cases (i) and (ii), from above equation, we obtain 
 \begin{align*}
    \frac{\pa}{\pa t} \int_{\Om} \frac{1}{2} (u-v)^2 dx \le 0.
 \end{align*}
 Therefore, the function $H: [0,T]\to \mb R$, defined by $H(t)= \int_{\Om} \frac{1}{2} (u-v)^2 dx $ is decreasing. Since $u$ and $v$ are distinct and $u(\cdot, 0)= v(\cdot, 0)$, we have 
  \[ 0< H(t)\le H(0)=0 \quad \mbox{ for }t\in[0,T], \]
  which is a contradiction. Thus $H(t)=0$ for all $t\in[0,T]$ and hence $u\equiv v$. \\
  To complete the proof, we need to show that $u\in C([0,T]; W^{1,p}_0(\Om))$ and \eqref{eq7} holds. From \eqref{eq42} and \eqref{eq45}, it is clear that $u\in L^\infty(0,T; W^{1,p}_0(\Om))$ and $u\in C([0,T]; L^2(\Om))$. Therefore, it follows that $u :t\in[0,T]\ra W^{1,p}_0(\Om)$ is weakly continuous and $u(\cdot, t_0)\in W^{1,p}_0(\Om)$ together with $\| u(\cdot, t_0)\|_{W^{1,p}_0(\Om)}\le \liminf_{t\ra t_0} \| u(\cdot, t)\|_{W^{1,p}_0(\Om)}$ for all $t_0\in [0,T]$ (and analogous result holds for $W^{1,q}_0(\Om)$ also). Multiplying \eqref{eq38} by $u^n -u^{n-1}$ and integrating over $\Om$ then summing from $N_2 \le N_1$, we get 
  \begin{align*}
  	\frac{\De_t}{2} \sum_{n=N_2}^{N_1} \int_{\Om} & \Big[\frac{u^n-u^{n-1}}{\De_t}\Big]^2 +\frac{1}{p} \left[ \int_{\Om} |\na u^{N_1}|^p -\int_{\Om} |\na u^{N_2-1}|^p \right] +\frac{1}{q} \left[ \int_{\Om} |\na u^{N_1}|^q -\int_{\Om} |\na u^{N_2-1}|^q \right] \\
  	&\quad -\frac{\vth}{1-\de} \left[ \int_{\Om} |u^{N_1}|^{1-\de} -\int_{\Om} |u^{N_2-1}|^{1-\de} \right] \le \sum_{n=N_2}^{N_1}\int_{\Om} g_{\De_t} (u^n-u^{n-1})dx.
  \end{align*}
 For $t_1\in[t_0,T]$, taking $N_2$ and $N_1$ such that $N_2\De_t \ra t_0$ and 
 $N_1\De_t \ra t_1$ as $\De_t\ra 0^+$, we obtain
 \begin{align}\label{eq39}
  \int_{t_0}^{t_1} &\int_{\Om}  \Big(\frac{\pa u}{\pa t}\Big)^2 dxdt +\frac{1}{p} \|u(\cdot,t_1)\|^p_{W^{1,p}_0(\Om)}  +\frac{1}{q}  \|u(\cdot,t_1)\|^q_{W^{1,q}_0(\Om)} 
   -\frac{\vth}{1-\de}  \int_{\Om} |u(x,t_1)|^{1-\de} dx \nonumber\\
  & \le \int_{t_0}^{t_1}\int_{\Om} g \frac{\pa u}{\pa t}dxdt +\frac{1}{p} \|u(\cdot,t_0)\|^p_{W^{1,p}_0(\Om)} +\frac{1}{q}  \|u(\cdot,t_0)\|^q_{W^{1,q}_0(\Om)}- \frac{\vth}{1-\de}  \int_{\Om} |u(x,t_0)|^{1-\de} dx.
 \end{align}
 Assume that $\|u(\cdot,t_1)\|_{W^{1,p}_0(\Om)} \ra l_1$ and $\|u(\cdot,t_1)\|_{W^{1,q}_0(\Om)} \ra l_2$, as $t_1\ra t_0^+$. Then, using  the fact that $u\in L^\infty(0,T; L^r(\Om))$ for all $r>1$ and dominated convergence theorem, passing on the limit $t_1\ra t_0^+$, we get 
    \begin{align}\label{eq36}
    \frac{1}{p} l_1^p  +\frac{1}{q}  l_2^q  \le \frac{1}{p} \|u(\cdot,t_0)\|^p_{W^{1,p}_0(\Om)}  +\frac{1}{q}  \|u(\cdot,t_0)\|^q_{W^{1,q}_0(\Om)}.
  \end{align}
Since $l_1\ge \|u(\cdot,t_0)\|_{W^{1,p}_0(\Om)}$ and $l_2\ge \|u(\cdot,t_0)\|_{W^{1,q}_0(\Om)}$, from \eqref{eq36}, we obtain $\|u(\cdot,t_1)\|^p_{W^{1,p}_0(\Om)}\ra \|u(\cdot,t_0)\|^p_{W^{1,p}_0(\Om)}$ as $t_1\ra t_0^+$ (and similar result in $W^{1,q}_0(\Om)$). Therefore, $u(t)\ra u(t_0)$ in $W^{1,p}_0(\Om)$, as $t\ra t_0^+$, which implies that $u$ is right continuous on $[0,T]$.
To prove the left continuity, we take $t_1>t_0$ and $0<k<t_1-t_0$. Set $\rho_k(u)(s):= \frac{u(s+k)-u(s)}{k}$, then taking it as a test function in the weak formulation of $(G_t)$ and using convexity, we deduce that 
 {\small  \begin{align*}
    \int_{t_0}^{t_1} &\int_{\Om} \frac{\pa u}{\pa t}\rho_k(u) dxdt +\frac{1}{kp}\int_{t_0}^{t_1} \int_{\Om} (|\na u(t+k)|^p -|\na u(t)|^p) +\frac{1}{kq}\int_{t_0}^{t_1} \int_{\Om} (|\na u(t+k)|^q -|\na u(t)|^q) \\
 	&\quad -\frac{\vth}{k(1-\de)}\int_{t_0}^{t_1} \int_{\Om} \big(u^{1-\de}(t+k)-u^{1-\de}(t)  \big)\ge \int_{t_0}^{t_1}\int_{\Om} g\rho_k(u)dxdt.
 \end{align*} }
 Now proceeding similar to \cite[page 5068]{badra}, as $k\ra 0^+$, we get
 \begin{align*}
 	 \int_{t_0}^{t_1} &\int_{\Om}  \Big(\frac{\pa u}{\pa t}\Big)^2 dxdt +\frac{1}{p} \|u(\cdot,t_1)\|^p_{W^{1,p}_0(\Om)}  +\frac{1}{q}  \|u(\cdot,t_1)\|^q_{W^{1,q}_0(\Om)} 
 	-\frac{\vth}{1-\de}  \int_{\Om} |u(x,t_1)|^{1-\de} dx \\
 	& \ge \int_{t_0}^{t_1}\int_{\Om} g \frac{\pa u}{\pa t}dxdt +\frac{1}{p} \|u(\cdot,t_0)\|^p_{W^{1,p}_0(\Om)} +\frac{1}{q}  \|u(\cdot,t_0)\|^q_{W^{1,q}_0(\Om)}- \frac{\vth}{1-\de}  \int_{\Om} |u(x,t_0)|^{1-\de} dx.
 \end{align*}
 From \eqref{eq39}, it is clear that the above inequality is in fact an equality. Therefore, using the fact that $t\mapsto\int_{\Om} u(t)^{1-\de}dx$ is continuous, we obtain $u\in C([0,T]; W^{1,p}_0(\Om))$. Moreover, \eqref{eq7} follows by taking $t_0=0$ in the above equality.\QED
 

 \section{Solution of problem $(P_t)$: Existence, uniqueness, regularity and asymptotic behaviour}
 In this section we obtain solution of problem $(P_t)$ under different growth conditions on the nonlinear term $f$.
 \subsection{Sub-homogeneous case}
Here, we study problem $(P_t)$ when the nonlinear term $f$ exhibit the sub-homogeneous growth with respect to the exponent $q$. First, we prove the existence result. \\
 \textbf{Proof of Theorem \ref{thm5}}: Let $T>0$ and $n\in\mb N$. Set $\De_t := T/n$. We follow the proof done in Section 2, to construct a sequence of solution $\{u_n\}\subset L^\infty(\Om)\cap W^{1,p}_0(\Om)$ of the following problem:
 \begin{align}\label{eq53}
 	u^n -\De_t \Big(\De_p u^n+\De_q u^n + \vth (u^n)^{-\de}\Big)= \De_t f(x, u^{n-1})+u^{n-1} \quad \mbox{in }\Om.
 \end{align}
 To start the iteration we take $u^0 =u_0\in\mc C_\de$. Applying Theorem \ref{thm1} for $\la=\De_t$ and $h=\De_t f(x,u^0)+u^0\in L^\infty(\Om)$, we get $u^1\in\mc C_\de\cap W^{1,p}_0(\Om)$ and satisfies \eqref{eq53}. Continuing the above procedure we get the existence of a sequence $\{u_n\}\in W^{1,p}_0(\Om)\cap \mc C_\de$. Furthermore, since $u_0\in\mc C_\de$, there exists $0<c_1<1<c_2$ such that $c_1 d\le u_0 \le c_2 d$ a.e. in $\Om$. Taking into account Corollary \ref{cor1}, we take $\rho>0$ small enough in Lemma \ref{lemsub} such that $ \tl u_\rho(x)\le u_0(x)$ in $\Om$ and it forms a subsolution of \eqref{eq53} (here we use the relation $f(x,s)\ge -L$). For the supersolution, we choose $M > \max\{1, \vth \}$ large enough in Proposition \ref{sup} such that $u_0\le v_M$ in $\Om$, where $v_M$ is the solution of problem $(\mc M)$ as in \eqref{eq52}. Set $\ul u(x):=\tl u_\rho$ and $\ov u=v_M$, then using weak comparison principle, we see that $\ul u \le u_n \le \ov u$ holds uniformly in $n\in\mb N$. While for the case $\de\ge 1$, we take $\ul u:=v_\rho$ as in the proof of Theorem \ref{thm4} with $\rho>0$ sufficiently small such that \eqref{eq78} holds, that is $\ul u\le u_0$. For the supersolution, we take $l=m$ and $M$ large enough in \eqref{eq77} such that the solution $\tl u_M:=\ov u$ satisfies \eqref{eq80}, that is $\ov u\ge u_0$. \\
 Let $u_{\De_t}$, $\tl u_{\De_t}$ be defined as in \eqref{eq30}. Set $u_{\De_t}(\cdot, t)= u_0(\cdot)$ for $t<0$ and $g_{\De_t}:= f(x,u_{\De_t}(t-\De_t))$ on $[0,T]$. Then, it is easy to verify that $u_{\De_t}$ and $\tl u_{\De_t}$ satisfies \eqref{eq31}. Due to the fact that $u_{\De_t}\in [\ul u, \ov u]$ and the map $t\mapsto f(x,t)$ is continuous on $[\ul u, \ov u]$, we get that $g_{\De_t}$ is bounded in $L^\infty(Q_T)$. Following the proof of Theorem \ref{thm2}, we get the existence of $u\in L^\infty(Q_T)\cap L^\infty(0,T; W^{1,p}_0(\Om))$ such that, up to a subsequence, 
  \begin{align*}
  	&u_{\De_t}, \tl u_{\De_t} \overset{\ast}{\rightharpoonup} u \in L^\infty(0,T; W^{1,p}_0(\Om)) \ \mbox{ and in }L^\infty(Q_T), \\
  	&\qquad   \frac{\pa \tl u_{\De_t}}{\pa t} \rightharpoonup \frac{\pa u}{\pa t} \ \ \mbox{ in }L^2(Q_T) \\
  	&u_{\De_t}^{-\de}\in L^\infty(0,T; W^{-1,p'}(\Om)) \quad \mbox{and } \|\tl u_{\De_t}-u_{\De_t} \|_{L^2(\Om)} \le C(\De_t)^\frac{1}{2}.
  \end{align*}
 Using Aubin-Lions-Simon lemma and proceeding similar to the proof of Theorem \ref{thm2}, for all $r>1$, we obtain 
 \begin{align}\label{eq54}
   u_{\De_t}, \tl u_{\De_t} \ra u \ \ \mbox{ in }L^\infty(0,T; L^r(\Om)) \quad \mbox{and } u\in C([0,T], L^r(\Om)).
 \end{align}
 Let $\omega>0$ be the Lipschitz constant of $f$ on $[\ul u, \ov u]$, then we have 
 \begin{align*}
    \| f(x, u_{\De_t}(t-\De_t)) -f(x, u(t)) \|_{L^2(\Om)} \le \omega \| u_{\De_t} (t-\De_t) -u(t) \|_{L^2(\Om)}
 \end{align*}
 and from \eqref{eq54}, we infer that $g_{\De_t}= f(x, u_{\De_t}(\cdot-\De_t)) \ra f(x,u)$ in $L^\infty(0,T; L^2(\Om))$. We can now complete rest of the proof proceeding similar to the proof of Theorem \ref{thm2}. The other part of the Theorem follows from the proof of \cite[Theorem 0.13, 5067]{badra}.\QED

\textbf{Proof of Proposition \ref{prop2}}: Let $u_0\in \ov{\mc D(\mc A)}^{L^\infty(\Om)}$, $\la>0$ and $f,g\in L^\infty(\Om)$. Invoking Theorem \ref{thm2}, we obtain the existence of unique solution $u,v\in W^{1,p}_0(\Om)\cap\mc C_\de\cap C_0(\ov\Om)$ of the problem
\begin{align}\label{eq62}
\begin{cases}
u+\la \mc A(u)=f \quad\mbox{in }\Om,\\
v+\la \mc A(v)=g \quad\mbox{in }\Om,
\end{cases}
\end{align}
respectively. It is clear that $u,v\in\mc D(\mc A)$. Set $w:=(u-v-\|f-g\|_{L^\infty(\Om)})^+$ and taking this as a test function in the weak formulation of \eqref{eq62}, we deduce that 
\begin{align*}
\int_{\Om} w^2+ \la \int_{\Om} (|\na u|^{p-2}\na u -|\na v|^{p-2}\na v)\na w+ \la \int_{\Om} (|\na u|^{q-2}\na u -|\na v|^{q-2}\na v)\na w \le 0.
\end{align*}
Invoking the inequalities used in the proof of theorem \ref{thm2} case (i) and (ii), we deduce that $u-v \le \|f-g\|_{L^\infty(\Om)}$. Reversing the role of $u$ and $v$, we obtain $\|u-v\|_{L^\infty(\Om)}\le \|f-g\|_{L^\infty(\Om)}$. This proves that the operator $\mc A$ is $m$-accretive in $L^\infty(\Om)$. Now, rest of the proof can be completed using \cite[Chap.4, Theorem 4.2 and Theorem 4.4]{barbu} or one can follow the steps of \cite[Proposition 0.2]{badra}. \QED

\textbf{Proof of Corollary \ref{cor3}}: On account of the assumption $u_0\in\mc D(\mc A)$, we deduce from proposition \ref{prop2} that
 \begin{align}\label{eq64}
 	\sup_{t\in [0,T]} \| u_t(\cdot, t) \|_{L^\infty(\Om)} \le e^{\omega T} \| \De_p u_0 +\De_q u_0 +\vth u_0^{-\de}+f(x,u_0)\|_{L^\infty(\Om)}.
 \end{align}
 Then, we see that the problem $(P_t)$ falls into the category of $(PS)$ with $b(x):=u_t(x,t)$, for  $\de>1$ (thanks to \eqref{eq64}). Using Theorem \ref{sob-int}, we get that $u(t)\in W^{1,m}_0(\Om)$ for all $m<\frac{p-1+\de}{\de-1}$ and its bound depends only on $m,p,\Om,N$ and the bounds of $u_0$ and $u^{-\de}$. This together with the fact $u\in\mc C_\de$ uniformly in $t\in[0,T]$, implies that $u \in L^\infty(0,T;W^{1,m}_0(\Om))$. Now, the claim of the corollary follows by interpolation argument, that is, $u \in C([0,T];W^{1,l}_0(\Om))$ for all $p\le l<m<\frac{p-1+\de}{\de-1}$, for $1<\de<2+1/(p-1)$. \QED

\textbf{Proof of Theorem \ref{thm6}}: Let $\ul u,\ov u\in W^{1,p}_0(\Om)\cap\mc C_\de\cap C(\ov\Om)$ be the sub and supersolution, as considered in theorem \ref{thm4}, to the problem
\begin{equation*}
 (P)\; \left\{\begin{array}{rllll}
 -\Delta_{p}u -\Delta_{q}u & = \vth \; u^{-\de}+ f(x,u), \; u>0 \text{ in } \Om, \\ u&=0 \quad \text{ on }  \pa\Om 
\end{array}
\right.
\end{equation*}
such that $\ul u\le u_0\le \ov u$. Let $u$ be the solution of problem $(P_t)$. Invoking Theorem \ref{thm5} again, we get the existence of $w_1$ and $w_2$, the unique solution of problem $(P_t)$ with initial datum $\ul u$ and $\ov u$, respectively. We claim that $\ul u, \ov u\in\ov{\mc D(\mc A)}^{L^\infty(\Om)}$. Let $h,g\in W^{-1,\frac{p}{p-1}}(\Om)$ be such that $\mc A\ul u:=h$ and $\mc A\ov u:=g$. Then, by the construction of $\ul u$ and $\ov u$, we have $h:=\mc A(\ul u)\le -L\le 0$ and $g:=\mc A(\ov u)\ge L+m \ov u^{q-1}\ge 0$. Set $h_n:=\max\{ h,-n\}$ and $g_n:=\min\{g,n\}$. Let $\{u_n\}$ and $\{v_n\}$ be two sequences in $\mc D(\mc A)$ such that $\mc Au_n=h_n$ and $\mc Av_n=g_n$. As a consequence of weak comparison principle, we obtain $\{u_n\}$ is nonincreasing while $\{v_n\}$ is nondecreasing. Since $h_n\ra h$ and $g_n\ra g$ in $W^{-1,\frac{p}{p-1}}(\Om)$, as $n\ra\infty$, we get $u_n\ra \ul u$ and $v_n\ra \ov u$ in $W^{1,p}_0(\Om)$. Therefore, up to a subsequence, $u_n\ra \ul u$ and $v_n\ra \ov u$ pointwise a.e. in $\Om$ and using Dini's theorem, we obtain $u_n\ra \ul u$ and $v_n\ra \ov u$ in $L^\infty(\Om)$. This proves the claim.\\
Next, from the fact that  $\ul u, \ov u\in\ov{\mc D(\mc A)}^{L^\infty(\Om)}$ and Theorem \ref{thm5}, we get $w_1,w_2\in C([0,T];C_0(\ov\Om))$. Furthermore, since $\ul u, \ov u\in\mc C_\de$, taking $\ul u^0=\ul u$ in the iterative scheme \eqref{eq53}, we obtain a nondecreasing sequence $\{\ul u^n\}$, where $\ul u^n$ is solution to the problem \eqref{eq53}, for $\De_t<1/\omega$ with $\omega>0$ as the Lipschitz constant of $f$ on $[\ul u, \ov u]$. Analogously, by taking $\ov u^0=\ov u$, we get a nonincreasing sequence $\{\ov u^n\}$, where  $\ov u^n$ is solution to the problem \eqref{eq53}. Let $\{u^n\}$ be the sequence obtained in \eqref{eq53} with $u^0=u_0$, then by the choice of $\De_t$, we can show that 
\begin{align}\label{eq61}
	\ul u^n \le u^n \le \ov u^n.
\end{align}
Following the proof of Theorem \ref{thm5} and using \eqref{eq61}, we deduce that $w_1(t)\le u(t) \le w_2(t)$. Also, we note that the map $t\mapsto w_1(x,t)$ is nondecreasing while $t\mapsto w_2(x,t)$ is nonincreasing. Therefore, $w_1(t)\ra w_1^\infty$ and $w_2(t)\ra w_2^\infty$ as $t\ra\infty$, where $w_1^\infty$ and $w_2^\infty$ is solution to the stationary problem $(P)$. Let $S(t)$ be the semigroup on $L^\infty(\Om)$ generated by the evolution equation $u_t+\la \mc A(u)=f(x,u)$, then we have
\begin{align*}
  w_1^\infty=\lim_{\tl t\ra\infty} S(\tl t+t)(\ul u)= S(t) \lim_{\tl t\ra\infty} S(\tl t)(\ul u)=S(t)\lim_{\tl t\ra\infty} w_1(\tl t)=S(t)w_1^\infty,
\end{align*}
and analogously, $w_2^\infty= S(t)w_2^\infty$. By the uniqueness result for problem $(P)$, we have $w_1^\infty=w_2^\infty=u_\infty\in C(\ov\Om)$. Therefore, by Dini's theorem, $w_1(t)\ra u_\infty$ and $w_2(t)\ra u_\infty$,  in $L^\infty(\Om)$ as $t\ra\infty$, then the required result of the Theorem follows from $w_1(t)\le u(t) \le w_2(t)$.\QED

\subsection{Super-homogeneous Case}
 In this subsection we discuss problem $(P_t)$ under the assumption (f3) on $f$. We first prove the local existence result using nonlinear semigroup theory of $m$-accretive operators.\\
 \textbf{Proof of Theorem \ref{thm7}}:  Let us fix $v\in L^\infty(Q_T)$ and consider the problem
  \begin{equation*}
  (B_t)\; \left\{\begin{array}{rllll}
  u_t-\Delta_{p}u -\Delta_{q}u -\vth \; u^{-\de} & = f(v), \; u>0 \text{ in } \Om\times (0,T), \\ u&=0 \quad \text{ on }  \pa\Om\times (0,T), \\
  u(x,0)&= u_0(x) \; \text{ in }\Om,
  \end{array}
  \right.
  \end{equation*}
 Set $\mc A u:= -\De_p u -\De_q u - \vth \; u^{-\de}$, then from the proof of Proposition \ref{prop2}, it is clear that the operator $\mc A$ is $m$-accretive in $L^\infty(\Om)$. Therefore,  $\mc A$ generates a contraction semigroup $S(t)$. From Theorem \ref{thm2}, there exists a unique solution $u\in C([0,T],L^\infty(\Om))$ and by nonlinear Semigroup theory (see \cite{barbu1}), the following representation formula holds
 \begin{align}\label{eq94}
 	u(t)= S(t)u_0+ \int_{0}^{t}S(t-\tau)f(v(\tau))d\tau.
 \end{align} 
Now, we will apply fixed point theorem to get a solution of problem $(P_t)$ from here.  Define $\mc F: C([0,T]; L^\infty(\Om))\subset L^\infty(Q_T)\to C([0,T]; L^\infty(\Om))$ as $\mc F(v)=u$, where $u$ is the solution to $(B_t)$. We consider the case $u_0\not\equiv 0$, and take $R_1>0$ such that $R_1= 2\| u_0 \|_{L^\infty(\Om)}$. For simplicity, denote $X= C([0,T];L^\infty(\Om))$ and $\|v\|_{X}= \max_{t\in[0,T]}\|v(t)\|_{L^\infty(\Om)}$. Then, for $v\in X$ with $\| v\|_{X}\le R_1$,  from \eqref{eq94}, we infer that 
 \begin{align*}
 	|\mc F(v(x,t))|= |u(x,t)| \le |S(t)u_0(x)|+\int_{0}^{t}|S(t-s)f(v(x,s))|ds \quad\mbox{for }(x,t)\in \Om\times(0,T),
 \end{align*}
which upon using the fact that $S(t)$ is a contraction semigroup and $f$ is bounded on the compact set ${\|v\|_{X}\leq R_1}$, we get 
 \begin{align*}
 	\| \mc F(v)\|_{X} \le \| u_0 \|_{L^\infty(\Om)} + T \| f(v) \|_{X} 
 	    \leq \| u_0 \|_{L^\infty(\Om)} + T C_1,
 \end{align*} 
 where $| f(s)|\leq C_1$ for $|s|\le R_1$. Then, we choose $\hat T$ small enough so that $ \hat T C_1 \leq \|u_0 \|_{L^\infty(\Om)}$ (note that choice of $T$ depends only on $u_0$ and $f$). Thus, 
  \begin{align*}
 	\| \mc F(v) \|_{X} \le 2 \| u_0 \|_{L^\infty(\Om)}=R_1
 \end{align*}
 with $X= C([0,\hat T], L^\infty(\Om))$. Next, noting the fact that $f$ is locally Lipschitz, for $v,w\in X$ such that $\| v\|_{X}, \|w\|_{X}\leq R_1$, we have
  \begin{align*}
  	\| f(v)- f(w) \|_{X} \leq \omega \|v-w \|_{X},
  \end{align*}
 where $\omega>0$ is the Lipschitz constant for $f$ on $|s|\le R_1$.  Therefore, for $v,w\in X$ with $\| v \|_{X},\| w \|_{X}\le R_1$, we get
 \begin{align*}
 	\| \mc F(v) - \mc F(w) \| \leq \max_{t\in[0,\hat T]}\int_{0}^{t} |S(t-s)f(v(s))-f(w(s))| \leq \hat T \omega \| v -w\|_{X},
 \end{align*}
 further, taking $\hat T>0$ small enough so that $\hat T \omega<1$, we get that $\mc F$ is a contraction. Therefore, by Banach fixed point theorem, there exists $u\in C([0,\hat T];L^\infty(\Om))$ such that $\mc F(u)=u$, that is, $u$ is a solution of $(P_t)$ in $Q_{T}$, for $0<T<\hat T$. \QED

 \begin{Lemma}\label{lem9}
  For fixed $\e>0$ small enough and $L>0$ large enough, there exists a solution $\ul u_{\e}$ in $C([0,T]:C_0(\ov\Om))$ to the following problem:
 	\begin{equation*}
 	(P_\e)\; \left\{\begin{array}{rllll}
 	u_t-\Delta_{p}u -\Delta_{q}u & = \vartheta \; (u+\e)^{-\de}-L, \; u>0 \text{ in } Q_T, \\ u&=0 \quad \text{ on }  \Ga_T, \\
 	u(x,0)&= u_0(x) \; \text{ in }\Om,
 	\end{array}
 	\right.
 	\end{equation*}
 Moreover, $\ul u_\e \leq u$ in $Q_T$, where $u$ is a weak solution to problem $(P_t)$ under the hypothesis of Theorem \ref{thm8}. 
 \end{Lemma} 
\begin{proof}
 We first note that the problem $(P_\e)$ does not contain singularity on the nonlinear term and it is bounded above by constant depending on $\e$ also. Therefore, standard existence and regularity theorem proves the first part of the lemma. The second part follows from the comparison principle. Indeed, taking $(\ul u_\e-u)^+$ as a test function, for  $\de<2+1/(p-1)$, we get 
 \begin{align*}
 \frac{1}{2}\frac{d}{dt}\int_{\Om^+}(\ul u_\e-u)^2 dx &\le -\ld -\De_p \ul u_\e +\De_p u, \ul u_\e-u \rd - \ld -\De_q \ul u_\e +\De_q u, \ul u_\e-u \rd \\
 & \ +\int_{\Om^+} \vartheta \; ((\ul u_\e+\e)^{-\de}-u^{-\de})(\ul u_\e-u)+(-L-f(x,u))(\ul u_\e-u)\le 0,
 \end{align*}
 where $\Om^+:=\Om\cap\{\ul u_\e\ge u \}$ and $-L\le f(x,u)$. Since $\ul u_\e(x,0)=u_0=u(x,0)$, we deduce that $\ul u_\e \leq u$ in $Q_T$.\QED
\end{proof}	
 \begin{Lemma}
 	Under the hypothesis of Theorem \ref{thm8}, there holds
 	\begin{align}\label{eq91}
 	J_\vth(u(t))= J_\vth(u_0)-\int_{0}^{t}\int_{\Om} u_t^2 dx ds.
 	\end{align}
 \end{Lemma}	
 \begin{proof}
  We note that due to the relation $\ul u_\e \le u$, from lemma \ref{lem9}, we can differentiate the singular term with respect to $t$. Therefore,
 	\begin{align*}
 	\frac{d}{dt}J_\vth(u(t))&= \int_{\Om} (|\na u|^{p-2}+|\na u|^{q-2})\na u \na u_t-\int_{\Om} (u^{-\de}+f(u))u_t \\
 	&=\int_{\Om}(-\De_p u -\De_q u - u^{-\de}-f(u))u_t \\
 	&=-\int_{\Om} u_t^2 dx.
 	\end{align*}
  Integrating the above expression from $0$ to $t$ we get the required result.\QED
 \end{proof}
Next, we prove the blow up result, the proof presented here is inspired by idea of \cite{anton} where $p(x)$-equation is considered with super-homogeneous type non-singular nonlinearity. \\
 \textbf{Proof of Theorem \ref{thm8}}:  Set $M(t)= \frac{1}{2}\int_{0}^{t} \| u(\cdot,s) \|_{L^2(\Om)}^2 ds$ for $t\ge 0$. Then, for a solution $u$ to problem $(P_t)$ and taking into account $\ul u_\e \le u$ (see lemma \ref{lem9}), it is easy to prove
 \begin{align}\label{eq92}
  M'(t)= \frac{1}{2}\| u(\cdot,t) \|_{L^2(\Om)}^2 \quad \mbox{and } M''(t)=  \int_{\Om} u u_t dx = -I_\vth(u).
 \end{align}
 Since $r\ge p$, there exists $\la>0$ such that 
   \begin{align*}
   	\frac{1}{r} \leq \la \leq \frac{1}{p}.
   \end{align*}
 From \eqref{eq91}, \eqref{eq92} and using the fact $J_\vth(u_0)\le 0$, we deduce that 
  \begin{align}\label{eq95}
  	\left(\frac{1}{p}-\la\right) \int_{\Om}|\na u|^p  +\left(\frac{1}{q}-\la\right) &\int_{\Om}|\na u|^q  +\left(\la-\frac{1}{1-\de}\right)\vth \int_{\Om}| u|^{1-\de}  +\int_{\Om}( \la f(u)u-F(u)) \nonumber\\
  	&\qquad\quad+ \int_{0}^{t}\int_{\Om} u_t^2 dxdt = J_\vth(u_0)+\la M''(t) \le \la M''(t).
  \end{align}
 For $\de=1$, the third term on the left is replaced by $\la \vth |\Om|-\vth\int_{\Om}\log |u|$ in the above. Now, we prove part (i) of the theorem.\\
 \textit{Proof of (i)}: We first prove $M''(t)>0$ i.e., $I_\vth(u(t))<0$ for $t>0$. On the contrary, assume
 there exists $t_0>0$ such that $I_\vth(u(t))<0$ for all $t<t_0$ and $I_\vth(u(t_0))=0$ (note that $I_\vth(u_0)<0$ and $I_\vth(u(t))$ is continuous in $t$). Since $\frac{1}{2}\frac{d}{dt}\|u\|_{L^2(\Om)}^2= -I_{\vth}(u)$ and $I_{\vth}(u(t))<0$ for all $t<t_0$, it follows that $u_t\not\equiv 0$ in $(0,t_0)$, that is,  $\int_{0}^{t_0}\int_{\Om} u_t^2 dx ds>0$. From \eqref{eq91}, we have $J_{\vth}(u(t_0))<\Theta_{\vth}<0$, for $\vartheta<\la_*$ (which gives us $u(t_0)\neq 0$). Noticing the fact that $I_\vth(u(t_0))=0$ and $u(t_0)\neq 0$, we get $u(t_0)\in\mc N_{\vth}$. Therefore, $J_\vth(u(t_0))\geq \Theta_{\vth}$, which is a contradiction. Hence, $I_{\vth}(u(t))<0$ holds for all $t\ge 0$. This implies that $M'(t)$ is nondecreasing in $t$, hence $M'(t)\ge M'(0)>0$ for all $t>0$. Using this together with Sobolev and H\"older inequality, for $\de<1$, we deduce that 
  \begin{align*}
   \left(\frac{1}{q}-\la\right) \int_{\Om}|\na u|^q    +\left(\la-\frac{1}{1-\de}\right)\vth \int_{\Om}| u|^{1-\de} &\geq \left(\frac{1}{q}-\la\right) \frac{M'(t)^\frac{q}{2}}{C_*^{q}}- \left(\frac{1}{1-\de}-\la \right)\vth \frac{M'(t)^\frac{1-\de}{2}}{C_*^{1-\de}} \\
   &\geq \frac{M'(t)^\frac{1-\de}{2}}{C_*^{1-\de}} \left(\frac{p-q}{pq}\frac{M'(0)^\frac{q-1+\de}{2}}{C_*^{q-1+\de}} -\vth \frac{r-1+\de}{r(1-\de)} \right)\\
   &\geq 0,
  \end{align*}
  for all $\vth<\vth_*:= \min\{ \la_*,\frac{r(p-q)(1-\de)2^{-(q-1+\de)/2}}{pq(r-1+\de)C_*^{q-1+\de}} \|u_0\|_{L^2(\Om)}^{q-1+\de} \}$, where $C_*>0$ is the suitable embedding constant (to avoid cumbersome we have used the same number). While for the case $\de=1$, we note that $\log |u| \leq |u|$, therefore proceeding similarly and replacing the term $M^\prime(t)^{(1-\de)/2}$ by $M^\prime(t)^{1/2}$ in above, we get a similar result.  Thus, from \eqref{eq95}, for all $\vth<\vth_*$, we obtain
   \begin{align}\label{eq97}
   	\int_{\Om}( \la f(u)u-F(u)) dx+\int_{0}^{t}\int_{\Om} u_t^2 dxdt \le \la M''(t).
   \end{align} 
 On the contrary assume that $T^*=\infty$, where
  \begin{align*}
  	T^*:=\sup\{t>0 \; : \; \| u(s) \|_{L^\infty(\Om)}<\infty \; \mbox{ for }s<t \}.
  \end{align*}
 Since $u$ is a non stationary solution of $(P_t)$, it follows from \eqref{eq97} and condition (f3) that there exist positive constants $c_1$ and $t_1$ such that $M''(t)\ge c_1$ and $M(t)\ge c_1$ for all $t\ge t_1$. This implies that $M'(t)\ra \infty$ as $t\ra\infty$.\\ 
 Now, we first consider the case $p>2$. Since $M'(t)\ra\infty$ as $t\ra\infty$, for every $\sigma\in(1,p/2)$, there exists $t_2>t_1$ such that 
 \begin{align}\label{eq98}
 1-\Big(\frac{2\sigma}{p}\Big)^{1/2} \ge \frac{M'(t_1)}{M'(t)} \quad \mbox{and }M(t_2)>0 \quad \mbox{ for all }t\ge t_2.
 \end{align}
 Moreover, from \eqref{eq97} and using H\"older inequality, for all $t\ge t_1$, we have
  \begin{align}\label{eq99}
   \big(M'(t)-M'(t_1)\big)^2 = \left(\int_{t_1}^{t} \int_{\Om} uu_t~dxdt\right)^2 &\leq \int_{t_1}^{t} \| u_t\|_{L^2{(\Om)}}^2 \int_{t_1}^{t}\| u\|_{L^2{(\Om)}}^2 \nonumber\\
   &\leq 2 \la M''(t) M(t) \leq \frac{2}{p} M''(t)M(t).
  \end{align}
 Taking into account \eqref{eq98} and \eqref{eq99}, we obtain
 \begin{align*}
 	\sigma M'(t)^2 \leq M''(t)M(t) \quad\mbox{for all }t\ge t_2.
 \end{align*}
 On integrating the above inequality, we get 
 \begin{align*}
 	M^{\sigma-1}(t)\ge \frac{M^{\sigma-1}(t_2)}{1-(t-t_2)(\sigma-1)M'(t_2)/M(t_2)} \ra\infty \quad\mbox{as }t\ra T^*_1:= t_2+ \frac{M(t_2)}{(\sigma-1)M'(t_2)},
 \end{align*}
 which yields a contradiction to our assumption $T^*=\infty$. Indeed,
  \begin{align*}
  	\infty> \frac{T^*_1|\Om|}{2} \sup_{t\in (0,T^*_1)} \|u(t)\|_{L^\infty(\Om)} \ge \frac{1}{2}\int_{0}^{t}\int_{\Om}u^2 dxdt =M(t)\ra\infty \quad \mbox{as }t\ra T^*_1.
  \end{align*} 
For the case $p\le 2$, we see that $r>2 \ge p$, therefore we can take $\la=\frac{1}{p}>\frac{1}{r}$ in \eqref{eq95} and from \eqref{eq97} using (f3), for all $\vth<\vth_*$, we deduce that 
 \begin{align*}
 	c_r\frac{r-p}{p} \int_{\Om}|u|^r dx \le \int_{\Om}( \la f(u)u-F(u)) dx \le \la M''(t).
 \end{align*}
 Therefore, for $t\ge t_3\ge t_2$ such that $M(t)$, $M'(t)$ and $M''(t)$ all are strictly positive,  and using H\"older inequality, we get 
 \begin{align*}
 	C_r M^\prime(t)^{r/2} \leq M''(t) \quad\mbox{for all }t\ge t_3,
 \end{align*}
 which upon integration gives us
  \begin{align*}
   M^\prime(t)^{r/2-1} \ge \frac{1}{M^\prime(t_3)^{1-r/2}-C_r(r/2-1)(t-t_3)}\ra\infty \quad\mbox{as }t\ra T^*_2:=t_3+\frac{2M'(t_3)^{(1-r/2)}}{C_r(r-2)}.
  \end{align*}
 This yields a contradiction to the assumption $T^*=\infty$, as in the case $p>2$.

 \textit{Proof of (ii)}: In this case, noticing $\de>1$, we see that the third term on the left of \eqref{eq95} is nonnegative. Therefore, we directly arrive at \eqref{eq97} for all $\vth>0$. Now, rest of the proof of the theorem, for non-stationary solution $u$ of problem $(P_t)$, follows similarly as in the case $\de <1$ by distinguishing the cases $p>2$ and $p\le 2$. \QED

 {\bf Acknowledgments}\\
 K. Sreenadh acknowledges the support through the Project:
 MATRICS grant MTR/2019/000121 funded by SERB, India.

\end{document}